\documentclass[reqno]{amsart}
 \newtheorem{thm}{Theorem}[section]
 
 \newtheorem{lem}[thm]{Lemma}
 
 \theoremstyle{definition}
 
 \theoremstyle{remark}
 \newtheorem{rem}[thm]{Remark}
 
 \numberwithin{equation}{section}

\usepackage{verbatim}
\usepackage{xcolor}
\usepackage{graphicx}
\usepackage{amssymb, amsmath, amsthm}
\usepackage{amsfonts}
\usepackage{amssymb}
\usepackage{float}
\usepackage{hyperref}
\usepackage{mathrsfs}
\usepackage{enumitem}

\newtheorem{theorem}{Theorem}

\newtheorem{lemma}[theorem]{Lemma}

\newcommand{\N}{\mathbb N}

\newcommand{\R}{\mathbb R}

\newcommand{\sgn}{\text{\rm sign}}

\newcommand{\PV}{\texttt{PV}}

\numberwithin{equation}{section}
\numberwithin{theorem}{section}

\numberwithin{figure}{section}
\usepackage{scalerel,stackengine}
\stackMath
\newcommand\reallywidehat[1]{%
\savestack{\tmpbox}{\stretchto{%
  \scaleto{%
    \scalerel*[\widthof{\ensuremath{#1}}]{\kern.1pt\mathchar"0362\kern.1pt}%
    {\rule{0ex}{\textheight}}
  }{\textheight}%
}{2.4ex}}%
\stackon[-6.9pt]{#1}{\tmpbox}%
}
\begin{document}

%
%
%
%
%
%
%
%
%

\title[Generalised 3D Muskat problem]
 {Existence, asymptotic behaviour and convergence of a generalised 3D Muskat problem in stable regime}

\author{Qasim Khan}

\address{Department of Mathematics and Information Technology\\
The Education University of Hong Kong}

\email{qasimkhan@s.eduhk.hk}

\author{Anthony Suen}

\address{Department of Mathematics and Information Technology\\
The Education University of Hong Kong}

\email{acksuen@eduhk.hk}

\author{Bao Quoc Tang}

\address{Department of Mathematics and Scientific Computing \\
University of Graz, Graz, Austria}

\email{quoc.tang@uni-graz.at}


\date{}



\begin{abstract}
We address a generalised three-dimensional $\alpha$-Muskat model that comes from the fluid interface problem given by two incompressible fluids with different densities in the stable regime. We establish local-in-time wellposedness when $\alpha\in[0,1)$ and also prove global-in-time existence for strong solutions when $\alpha\in[0,\frac{1}{2})$ with initial data controlled by explicit constants. We obtain maximum principles for the $L^{\infty}$-norms of both the solutions and their gradients, and we further acquire the corresponding decay rates of these $L^{\infty}$-norms. Finally, some convergence results for strong solutions as $\alpha\to0^+$ are also proved. 
\end{abstract}

\maketitle
\section{Introduction}\label{introduction}

This paper studies a generalised model for a free boundary problem in three dimensional space. Such model is motivated by the original Muskat problem which describes the evolution of the interface between two fluids under the influence of gravity and capillarity \cite{Musk37}. More precisely, we address the following generalised system for modelling the motion of fluids:
\begin{align}
\rho_t + \nabla \cdot  (u\rho) &= 0, \label{general abstract active scalar eqn nondiffusive modified 1}\\
u + \nabla P &= - (0,0,(-\Delta)^\frac{\alpha}{2}\rho), \label{general abstract active scalar eqn nondiffusive modified 2}\\
\nabla \cdot u &= 0 \label{general abstract active scalar eqn nondiffusive modified 3},
\end{align}
where $u$ is the velocity, $P$ is the pressure, $\rho$ is the density and $\alpha\in[0,1)$ is a non-negative parameter. Equations \eqref{general abstract active scalar eqn nondiffusive modified 1} and \eqref{general abstract active scalar eqn nondiffusive modified 3} are respectively the conservation of mass and incompressibility condition on the fluid, while equation \eqref{general abstract active scalar eqn nondiffusive modified 2} is known as the Darcy's law when $\alpha=0$. 

Denoting the interface
that separates the space in two domains $\Omega_1$ and $\Omega_2$ by $x_3 = f (x_1, x_2, t)$, we consider the density $\rho = \rho(x_1, x_2, x_3, t)$ to be the following step function
\begin{equation}\label{step function for rho}
\rho(x, x_1, x_2, x_3, t) =
\begin{cases}
\rho_1 , & x \in \Omega_1(t) = \{ x_3 > f(x_1, x_2, t) \} \\
\rho_2 , & x \in \Omega_2(t) = \{ x_3 < f(x_1, x_2, t) \},
\end{cases}
\end{equation}
where $\rho_1$ and $\rho_2$ are positive constant densities. By applying the curl operator twice and making use of the incompressibility of the fluid, we have ${\rm curl }\,{\rm curl } u = -\Delta u$. {Hence we can express the velocity in terms of the density by the equation
\begin{align}\label{eqn between velocity and density}
u[\rho]=u = (-\partial_{x_1}\Lambda^{-2+\alpha}\partial_{x_3}\rho, -
\partial_{x_2}\Lambda^{-2+\alpha}\partial_{x_3}\rho,
(\partial_{x_1}^2+\partial_{x_2}^2)\Lambda^{-2+\alpha}\rho),
\end{align}
with $\Lambda=(-\Delta)^\frac{1}{2}$. Applying the double-curl
vector identity to \eqref{general abstract active scalar eqn nondiffusive modified 2} yields 
$$-\Delta u = \nabla(\nabla \cdot F_\rho) - \Delta
F_\rho,$$ 
where $F_\rho = (0,0,-\Lambda^\alpha \rho)$.} Then the interface $f$ satisfies the following contour equation (see \cite{CG07} for a detailed derivation for the case $\alpha=0$), which we refer it as {\it generalised $\alpha$-Muskat problem} or simply {\it generalised Muskat problem}:
\begin{equation}\label{contour eqn}
\begin{split}
\frac{\partial f(x,t)}{\partial t} &= \frac{\rho_2 - \rho_1}{4\pi} \, \PV \, \int_{\mathbb{R}^2} \frac{ \left( \nabla f(x,t) - \nabla f(x - y, t) \right) \cdot y }{ \left[ |y|^2 + (f(x,t) - f(x - y, t))^2 \right]^{\frac{3 + \alpha}{2}}} \, dy, \\
f(x,0) &= f_0(x), \quad x=(x_1, x_2) \in \mathbb{R}^2.
\end{split}
\end{equation}
Here $\PV$ denotes the principle value. Throughout this work, we shall assume that the generalised Muskat problem is in the {\it stable regime}, in the sense that the denser fluid is below the other fluid. More precisely, the constant densities $\rho_1$ and $\rho_2$ satisfy
\begin{align}\label{stable case assumption}
\rho_2>\rho_1,
\end{align}
and for simplicity, we choose $(\rho_2 - \rho_1)/4\pi=1$ as in \eqref{contour eqn} unless otherwise stated. 

When $\alpha=0$, the system \eqref{general abstract active scalar eqn nondiffusive modified 1}-\eqref{general abstract active scalar eqn nondiffusive modified 3} describes the dynamics of the interface in between two incompressible fluids in porous media in three dimensional space, in particular it was used for addressing the encroachment of water into oil in a porous medium \cite{Musk37}. There is an extensive literature on the topic of Muskat problem and we shall give a brief review on the relevant results. When the surface tension is taken into account via the Laplace-Young condition, it was shown by Escher and Simonett \cite{ES97} that there exist classical solutions to the Muskat problem in the two-dimensional case. Later in \cite{EM11}, Escher and Matioc addressed the Muskat problem in a periodic geometry
and incorporate capillary and gravity in the setting, and they obtained local well-posedness for H\"{o}lder solution. Escher et al. \cite{EMV12} conducted a similar study for the case of three interfaces by reducing the problem to an abstract evolution equation. In the scenario where there is no surface tension, Siegel, Caflisch and Howison \cite{SCH04} proved that for the two-dimensional case, the problem is ill-posed when the higher-viscosity fluid contracts. Conversely, they established the global-in-time existence of solutions for small initial data when the higher-viscosity fluid expands. In \cite{AM04}, Ambrose proved short-time well-posedness for the problem with two fluids under the condition that the initial data must satisfy a sign condition which is a generalisation of a condition of Saffman and Taylor \cite{ST58}.

The free boundary problems given by fluids with different constant densities $\rho_1$ and $\rho_2$ as in \eqref{step function for rho} in the case $\alpha  = 0$ have also been extensively studied by various researchers. In the stable regime for which $\rho_1>\rho_2$, C\'ordoba and Gancedo \cite{CG07, cordoba2009maximum} proved that the Muskat problem is well-posed and obtained local-in-time existence and uniqueness for it. In the work \cite{constantin2016muskat}, Constantin et al. showed global-in-time existence for strong and weak solutions with initial data $f_0$ in \eqref{contour eqn} controlled by explicit constants, and the results were later extended by Gancedo et al. \cite{GGPS19} to the more general case with density-viscosity jump. Moreover in \cite{GL22}, Gancedo and Lazar showed that the three-dimensional problem is globally well-posed in the critical Sobolev space $\dot{H}^2\cap\dot{W}^{1,\infty}$ with arbitrary large slope, and by applying the modulus of continuity techniques from \cite{KNV07} for quasi-geostrophic equation, Cameron et al. \cite{CCH26} recently obtained local well-posedness results in any dimension for initial data within the space $W^{1,\infty}$. In \cite{FGSV12}, Friedlander et al. considered a singularly modified version of the two-dimensional Muskat problem and showed that the system has a unique local-in-time patch-type solution $\rho$ given by \eqref{step function for rho} with a corresponding smooth interface $f\in L^\infty(0,T;H^s(\R))$ for $s\ge4$ and $T>0$. In the unstable regime $\rho_1>\rho_2$, it was shown in \cite{SCH04} that in the two-dimensional case, there are singular solutions that start off with smooth initial data but develop a point of infinite curvature at finite time, hence the problem is ill-posed in the sense that singularities can form in
an arbitrarily short time for arbitrarily small initial data in terms of Sobolev norm. In \cite{CG07}, for the three-dimensional setting, the problem is ill-posed that there exists a solution $f$ such that its $H^s$ norm with $s>\frac{3}{2}$ blows up in finite time. Such solution was constructed using the rescaled global solution for the two-dimensional stable case. In particular, C\'ordoba et al. \cite{CGJ17} showed that there exist solutions of the Muskat problem which shift stability regimes in the following sense: they start stable, then become unstable, and finally return back to the stable regime. Castro, C\'ordoba and Faraco \cite{CCF21} later proved the existence of mixing solutions of the incompressible porous media equation for
all Muskat type $H^5$ initial data in the fully unstable regime.

There is also a considerably rich literature dedicated to the study of the regularity of solutions for the Muskat problem. Using the Cauchy-Kowalewski method, Castro, C\'ordoba and Gancedo \cite{CCG10} showed that there exist analytic solutions to the problem provided that the initial data are analytic functions. In terms of critical regularity, for the two-dimensional case, Alazard and Nguyen \cite{AN23} obtained well-posedness on the endpoint Sobolev space of $L^2$ functions with three-half derivative in $L^2$, and such result is indeed optimal with respect to the scaling of the equation. Regarding weak solutions, Constantin et al. \cite{CCG13} obtained global existence results for Lipschitz continuous solutions with initial data that satisfy $\|f_0\|_{L^\infty}<\infty$ and $\|\partial_xf_0\|_{L^\infty}<1$. In a more recent work \cite{JKM21}, Jacobs, Kim and M\'{e}sz\'{a}ros introduced a framework to approximate solutions of the Muskat problem with surface tension, and they constructed weak solutions via the convergence of approximate solutions. 

In this current work, our goal is to investigate various properties of interface $f$ to the generalised Muskat problem \eqref{contour eqn} with the presence of parameter $\alpha\in[0,1)$. In view of the system \eqref{general abstract active scalar eqn nondiffusive modified 1}-\eqref{general abstract active scalar eqn nondiffusive modified 3}, the fractional derivative $\Lambda^{\alpha}$ is inserted in the constitutive law \eqref{eqn between velocity and density} that is similar to the approach used in \cite{FGSV12} in two-dimensional case. Our current work is therefore the first one for addressing a generalised Muskat problem in the three-dimensional case. It is also worth pointing out that with the presence of $\Lambda^{\alpha}$, the operator $\rho\mapsto u[\rho]$ as given by \eqref{eqn between velocity and density} becomes a pseudo-differential operator of order $\alpha$ with Fourier multiplier symbol explicitly given by
\begin{align*}
{\hat{u}(k)= \left( k_1k_3|k|^{\alpha-1}\hat{\rho}, k_2k_3|k|^{\alpha-1}\hat{\rho},-(k_1^2+k_2^2)|k|^{\alpha-2}\hat{\rho} \right),}
\end{align*}
where $k=(k_1,k_2,k_3)$ is the frequency variable. We note that $\hat{u}(k)$ is even in $k$, which leads to results that are in sharp contrast to those obtained for the singular surface quasi-geostrophic equations \cite{CCCG12}. The even nature of the symbol relating the active scalar and the drift velocity may give rise to various ill/well-posedness results; see for example \cite{FV11, FS19, FS23, FSW24, Sh11} and the references therein.

The claimed results of this paper will be detailedly discussed in the subsequent sections. In Section~\ref{local-in-time sec}, in the regime $\alpha\in[0,1)$, we obtain local-in-time existence and uniqueness for the interface $f$ in $H^k(\R^2)$ with $k\ge4$ using energy estimates. In Section~\ref{global-in-time existence sec}, we further obtain global-in-time existence of unique $f$ if initially $f_0$ is bounded by $\|f_0\|_{1}<k_0(\alpha)$ with $\alpha\in[0,\frac{1}{2})$ and some constant \footnote{The existence of the constant $k_0(\alpha)$ will be shown in Appendix~\ref{existence of k0 and estimates on principal values sec}.} $k_0(\alpha)>0$, where $\|f_0\|_{1}$ is defined by \eqref{def of fourier norm} via Fourier transform. In Section~\ref{max principle and long time section}, we prove some results regarding the maximum principle for the interface $f$ and further address the long-time behaviour for $f$. Finally in Section~\ref{convergence for alpha section}, we address the convergence of $f$ as $\alpha\to0^{+}$. 


\section{Local-in-time well-posedness}\label{local-in-time sec}

In this section, we prove local-in-time existence and uniqueness for the interface $f$ using energy estimates. For $k\in\N$ and $\delta\in(0,1)$, the norms of $H^k$ and $C^{k,\delta}$ are defined by
\begin{align*}
\|f\|^2_{H^k}&:=\|f\|^2_{L^2}+\|\Lambda^k f\|^2_{L^2},\\
\|f\|_{C^{k,\delta}}&:=\|f\|_{C^k}+\max_{i+j=k}\max_{x\neq y}\frac{|\partial_{x_1}^i\partial_{x_2}^jf(x)-\partial_{x_1}^i\partial_{x_2}^jf(y)|}{|x-y|^\delta}.
\end{align*}

The main theorem in this section is given as follows.

\begin{thm}\label{local existence thm}
Let $f_0\in H^k(\R^2)$ for $k\ge4$. Then for any $\alpha\in[0,1)$, there exists a time $T>0$ such that the contour equation \eqref{contour eqn} possess a unique solution in $C^1([0,T];H^k(\R^2))$ with $f(x,0)=f_0(x)$.
\end{thm}

\begin{proof}
We apply energy method for proving Theorem~\ref{local existence thm}. Throughout this proof, $C$ shall denote a positive and sufficiently large constant whose value may change from line to line.

We multiply the contour equation \eqref{contour eqn} by $f$ and integrate to obtain
\begin{equation*}
\begin{split}
\frac{1}{2} \frac{d}{dt} \left\|f\right\|_{L^2}^2(t)&=\int_{ \mathbb{R}^2 } f(x) \PV \int_{\mathbb{R}^2 } \frac{\left( \nabla f(x) - \nabla f(x - y)\right)\cdot y }{ \left[ |y|^2+(f(x)-f(x-y) )^2 \right]^{\frac{3 + \alpha}{2}} } dxdy\, \quad \quad \quad  \\
&= \int_{ \mathbb{R}^2 } f(x) \int_{ |y| < 1 } \frac{\left( \nabla f(x) - \nabla f(x - y)\right)\cdot y }{ \left[ |y|^2+(f(x)-f(x-y) )^2 \right]^{\frac{3 + \alpha}{2}} } dxdy\,   \\
&\qquad+ \PV \int_{ \mathbb{R}^2 } f(x) \int_{ |y| > 1 } \frac{ \nabla f(x)\cdot y }{ \left[ |y|^2+(f(x)-f(x-y) )^2 \right]^{\frac{3 + \alpha}{2}} } dxdy\,    \\
&\qquad- \PV \int_{ \mathbb{R}^2 } f(x) \int_{ |y| > 1 } \frac{ \nabla f(x-y)\cdot y }{ \left[ |y|^2+(f(x)-f(x-y) )^2 \right]^{\frac{3 + \alpha}{2}} } dxdy\, \\
&\triangleq \mathbf{I}_1 + \mathbf{I}_2 + \mathbf{I}_3.
\end{split}
\end{equation*}
We estimate the terms $\mathbf{I}_1$, $\mathbf{I}_2$ and $\mathbf{I}_3$ as follows. By the fundamental theorem of calculus, we have
\begin{equation*}
\begin{split}
\mathbf{I}_1&\leq C \int_{ 0 }^{1} ds  \int_{ |y| < 1 }  {|y|^{2 -(3 +\alpha)} }  \ \int_{ \mathbb{R}^2 } \frac{|f(x)| \, |\nabla^2 f(x+(s-1)y)|} {\left[ 1+\left(f(x)-f(x-y)\right)^2|y|^{-2}\right]^{\frac{3+\alpha}{2}}} dxdy\\
& \leq C \ \left\|f\right\|_{L^2} \ \sum_{i+j=2} \left\|\partial_{x_1}^i\partial_{x_2}^j f\right\|_{L^2}\int_{ 0 }^{1} ds  \int_{ |y| < 1 }  {|y|^{-1-\alpha} }dy   \\
&\leq \frac{ C\left\|f\right\|_{H^2}^2 }{ 1 - \alpha }. 
\end{split}
\end{equation*}
Integrating by parts, the term $\mathbf{I}_2$ can be bounded by
\begin{equation*}
\begin{split}
\mathbf{I}_2&\leq \frac{3+\alpha}{2}  \int_{ |y| > 1 } \int_{ \mathbb{R}^2 }  |f(x)|^2 \cdot \frac{\left(f(x)-f(x-y)\right)\cdot \left(\nabla f(x)-\nabla f(x-y)\right)\cdot y} {\left[ |y|^2+(f(x)-f(x-y) )^2 \right]^{\frac{5+\alpha}{2}}} dxdy\\
& \leq C  \int_{ |y| > 1 } |y|^{-(5+\alpha)+1} \int_{ \mathbb{R}^2 } |f(x)|^2  \cdot \frac{|f(x)-f(x-y)|\cdot |\nabla f(x)-\nabla f(x-y)|} {\left[ 1+\left(f(x)-f(x-y) \right)^2 |y|^{-2}\right]^{\frac{5+\alpha}{2}}} dxdy\\
 &\leq C \left\|f\right\|_{L^\infty}\left\|f\right\|_{H^2}^2    \ \int_{ |y| < 1 } |y|^{-(4+\alpha)} dy\\
 &\leq \frac{C\left\|f\right\|_{L^\infty}\left\|f\right\|_{H^2}^2}{2+\alpha}.
\end{split}
\end{equation*}
Integrating by parts on $\mathbf{I}_3$, we also have
\begin{equation*}
\begin{split}
\mathbf{I}_3&=  \int_{ |y| > 1 } \int_{ \mathbb{R}^2 } f(x)f(x-y) \cdot \frac{|y|^2 -2\left(f(x)-f(x-y)\right)^2} {\left[ 1+\left(f(x)-f(x-y)\right)^2|y|^{-2}\right]^{\frac{5+\alpha}{2}}} dxdy\\
& \qquad+(3+\alpha) \ \int_{ |y| > 1 } \int_{ \mathbb{R}^2 } f(x)f(x-y) \cdot \frac{(f(x)-f(x-y))\cdot (\nabla f(x-y))\cdot y} {\left[ 1+\left(f(x)-f(x-y)\right)^2|y|^{-2}\right]^{\frac{5+\alpha}{2}}} dxdy\\
& \qquad-\int_{ |y| = 1 } \int_{ \mathbb{R}^2 } f(x)f(x-y) \cdot \frac{|y|^2 } {\left[ 1+\left(f(x)-f(x-y)\right)^2|y|^{-2}\right]^{\frac{3+\alpha}{2}}}dxdS \\
 &\leq C \left(\left\|f\right\|_{L^\infty} +1\right) \     \left\|f\right\|_{H^1}^2.
\end{split}
\end{equation*}
Combining the estimates on $\mathbf{I}_1$, $\mathbf{I}_2$ and $\mathbf{I}_3$ and applying Sobolev inequalities, we conclude that
\begin{align}\label{L2 estimate}
\frac{d}{dt}  \left\|f\right\|_{L^2}^2(t) \leq C (\left\|f\right\|_{H^2}^3(t)+1).
\end{align}

Next, in order to obtain higher order estimates, for $1\le k\le4$, we apply $\partial_{x_1}^k$ on \eqref{contour eqn}, multiply it by $\partial_{x_1}^k f$ and integrate to get
\begin{equation*}
\frac{1}{2} \frac{d}{dt} \left\|\partial_{x_1}^k f\right\|_{L^2}^2(t) = \mathbf{I}_{4} + \mathbf{I}_{5},
\end{equation*}
where
\begin{align*}
\mathbf{I}_{4}&=\int_{\R^2}\partial_{x_1}^kf(x)\PV\int_{\R^2}(\nabla\partial_{x_1}^kf(x)-\nabla\partial_{x_1}^kf(x-y))\cdot y A_\alpha(x,y)dydx,\\
\mathbf{I}_{5}&=\sum_{i=1}^k{k\choose i}\int_{\R^2}\partial_{x_1}^k f(x)\int_{\R^2}(\nabla\partial_{x_1}^if(x)-\nabla\partial_{x_1}^if(x-y))\cdot y \partial_{x_1}^{k-i}A_\alpha(x,y)dydx
\end{align*}
and
\begin{align*}
A_\alpha(x,y)=\left[ |y|^2 + (f(x) - f(x - y))^2 \right]^{-\frac{3 +\alpha}{2}}.
\end{align*}
To estimate $\mathbf{I}_{4}$, we first rewrite it as follows:
\begin{equation*}
\begin{split}
\mathbf{I}_4 &= \int_{\mathbb{R}^2}\partial_{x_1}^k f(x) \PV  \int_{\mathbb{R}^2} \frac{\nabla \partial_{x_1}^k \ f(x) \cdot y}{\left[ |y|^2 + (f(x) - f(x - y))^2 \right]^{\frac{3 +\alpha}{2}}} \, dydx\\
&\qquad- \int_{\mathbb{R}^2} \partial_{x_1}^k f(x) \PV \ \int_{\mathbb{R}^2} \frac{\nabla \partial_{x_1}^k f(x) \cdot (x - y)}{\left[ |y|^2 + \left(f(x) - f(x - y)\right)^2 \right]^{\frac{3 + \alpha}{2}}} \, dydx\\
&\triangleq \mathbf{J}_1+\mathbf{J}_2.
\end{split}
\end{equation*}
To bound $\mathbf{J}_1$, 
\begin{equation*}
\begin{split}
\mathbf{J}_1 &= \frac{3 + \alpha}{2} \int_{\mathbb{R}^2} | \partial_{x_1}^k f(x) |^2 \PV \int_{\mathbb{R}^2}  \frac{(f(x)-f(x-y))(\nabla f(x)-\nabla f(x-y))\cdot y }{\left[ 1+\left(f(x)-f(x-y)\right)^2|y|^{-2}\right]^{\frac{5 +\alpha}{2}}} dydx \\
&\le \frac{3+\alpha}{2}\|f\|_{C^1}\|\partial_{x_1}^k f\|^2_{L^2}+\frac{3+\alpha}{2}M_\alpha(f)\|\partial_{x_1}^k f\|^2_{L^2},
\end{split}
\end{equation*}
where \( M_\alpha(f) \) is given by
\begin{equation}\label{def of M alpha (f)}
M_\alpha(f) = \max_x\Big| \PV \int_{|y|<1}  \frac{\left(f(x)-f(x-y)\right)\cdot \left(\nabla f(x)-\nabla f(x-y)\right)\cdot y}{\left[ |y|^2 + \left(f(x) - f(x - y)\right)^2 \right]^{\frac{5 + \alpha}{2}}}dy\Big|.
\end{equation}
Following the similar method in \cite{CG07}, there exists $\delta>\max\{0,\alpha\}$ such that
\begin{align*}
M_\alpha(f)&\le \|f\|^2_{C^2}\max_{x}\Big|\int_{|y|<1}\frac{|y|^{-1}}{\left[ 1 + ((f(x) - f(x - y))|y|^{-1})^2 \right]^{\frac{5 + \alpha}{2}}}dy\Big|\\
&\qquad+\|f\|_{C^1}\|f\|_{C^{2,\delta}}\Big|\int_{|y|<1}|y|^{-2-\alpha+\delta}dy\Big|+C\|f\|^2_{C^1}\|f\|^2_{C^2}\\
&\le C( \left\|f\right\|_{C^2}^2 + \left\|f\right\|_{C^1}\left\|f\right\|_{C^{2,\delta}}   + \left\|f\right\|_{C^{1}}^2\left\|f\right\|_{C^2}^2),
\end{align*}
and hence
\begin{align*}
\mathbf{J}_1\le C(\|f\|^4_{C^{2,\delta}}+1)\|\partial^k_{x_1}f\|^2_{L^2}.
\end{align*}
To bound $\mathbf{J}_2$, we integrate by parts to rewrite $\mathbf{J}_2$ as follows:
\begin{align*}
\mathbf{J}_2&=\int_{\R^2}\partial_{x_1}^kf(x)\PV\int_{\R^2}\frac{\nabla_y(\partial_{x_1}^kf(x)-\partial_{x_1}^kf(y))\cdot(x-y)}{\left[ |x-y|^2 + \left(f(x) - f(y)\right)^2 \right]^{\frac{3 + \alpha}{2}}}dydx\\
&=-\frac{1}{2}\int_{\R^2}\int_{\R^2}\frac{(\partial_{x_1}^k(f(x)-\partial_{x_1}^k(f(y))^2}{\left[ |x-y|^2 + \left(f(x) - f(y)\right)^2 \right]^{\frac{3 + \alpha}{2}}}dydx+\mathbf{J}_3,
\end{align*}
where 
\begin{align*}
\mathbf{J}_3=(3+\alpha)\int_{\R^2}\partial_{x_1}^kf(x)\int_{\R^2}(\partial_{x_1}^kf(x)-\partial_{x_1}^kf(y))K_{\alpha}(x,y)dydx
\end{align*}
and
\begin{align*}
K_{\alpha}(x,y)=\frac{(f(x)-f(y))(f(x)-f(y)-\nabla f(y)\cdot(x-y))}{\left[ |x-y|^2 + \left(f(x) - f(y)\right)^2 \right]^{\frac{5 + \alpha}{2}}}.
\end{align*}
We further split $\mathbf{J}_3$ as 
\begin{align*}
\mathbf{J}_3=\mathbf{L}_1+\mathbf{L}_2+\mathbf{L}_3,
\end{align*}
with $\mathbf{L}_1$, $\mathbf{L}_2$, $\mathbf{L}_3$ being defined by
\begin{align*}
\mathbf{L}_1&=(3+\alpha)\int_{\R^2}|\partial_{x_1}^kf(x)|^2\PV\int_{\R^2}K_{\alpha}(x,y)dydx,\\
\mathbf{L}_2&=-(3+\alpha)\PV\int_{\R^2}\int_{\R^2}\partial_{x_1}^kf(x)\partial_{x_1}^kf(y)(K_{\alpha}(x,y)-D_{\alpha}(x,y))dydx,\\
\mathbf{L}_3&=-(3+\alpha)\int_{\R^2}\int_{\R^2}\partial_{x_1}^kf(x)\partial_{x_1}^kf(y)D_{\alpha}(x,y)dydx
\end{align*}
and
\begin{align*}
D_{\alpha}(x,y)=\frac{(f(x)-f(y)-\nabla f(y)\cdot(x-y)-\frac{1}{4}(x-y)\cdot(\nabla^2f(x)+\nabla^2f(y))\cdot(x-y))}{\left[ |x-y|^2 + \left(f(x) - f(y)\right)^2 \right]^{\frac{5 + \alpha}{2}}}.
\end{align*}
Notice that $K_{\alpha}(y,x)-D_{\alpha}(y,x)=-(K_{\alpha}(x,y)-D_{\alpha}(x,y))$, hence by interchanging $x$ and $y$, we have $\mathbf{L}_2=0$. And for $\mathbf{L}_1$, it can be bounded in the same way as for $\mathbf{J}_1$ to give
\begin{align*}
\mathbf{L}_1\le C(\|f\|^4_{C^{2,\delta}}+1)\|\partial^k_{x_1}f\|^2_{L^2}.
\end{align*}
To bound the term $\mathbf{L}_3$, since for $\delta>\max\{0,\alpha\}$ and $\alpha\in[0,1)$, we have
\begin{align*}
&\int_{\R^2}|\partial_{x_1}^kf(x)|^2\int_{\R^2}|f(x)-f(y)||D_{\alpha}(x,y)|dydx\\
&=\int_{\R^2}|\partial_{x_1}^kf(x)|^2\int_{\R^2}|f(x)-f(x-y)||D_{\alpha}(x,x-y)|dydx\\
&\le C\|f\|_{C^1}\|f\|_{C^{2,\delta}}\|\partial_{x_1}^kf\|^2_{L^2}\left(\int_{|y|<1}|y|^{-2-\alpha+\delta}dy+\int_{|y|>1}|y|^{-3-\alpha}dy\right)\\
&\le C\|f\|_{C^1}\|f\|_{C^{2,\delta}}\|\partial_{x_1}^kf\|^2_{L^2},
\end{align*}
hence by symmetry,
\begin{align*}
\mathbf{L}_3\le C\|f\|_{C^1}\|f\|_{C^{2,\delta}}\|\partial_{x_1}^kf\|^2_{L^2}.
\end{align*}
We conclude that
\begin{equation}\label{bound on I4}
\mathbf{I}_4 \leq C ( \left\|f\right\|_{C^{2+ \delta}}^4  + 1 ) \left\|\partial_{x_1}^4 f \right\|_{L_2}^2.
\end{equation}
The term $\mathbf{I}_5$ can be bounded in the same way as given in \cite{CG07}, which also gives
\begin{equation}\label{bound on I5}
\mathbf{I}_5\le C (\|f\|_{C^{2+ \delta}}^4+1)\left\|\partial_{x_1}^4 f \right\|_{L_2}^2.
\end{equation}
Combining \eqref{bound on I4} and \eqref{bound on I5}, for all $1\le k\le 4$,
\begin{align}\label{Hk estimate 1}
\frac{d}{dt}  \left\|\partial_{x_1}^4 f\right\|_{L^2}^2(t) \leq C (\|f\|_{C^{2+ \delta}}^4+1)\left\|\partial_{x_1}^4 f \right\|_{L_2}^2,
\end{align}
and in a similar way, we also have
\begin{align}\label{Hk estimate 2}
\frac{d}{dt}  \left\|\partial_{x_2}^4 f\right\|_{L^2}^2(t) \leq C (\|f\|_{C^{2+ \delta}}^4+1)\left\|\partial_{x_2}^4 f \right\|_{L_2}^2,
\end{align}

Since $\|f\|^2_{H^4}=\|f\|^2_{L^2}+\|\partial_{x_1}^4 f\|_{L^2}^2+\|\partial_{x_2}^4 f\|_{L^2}^2$, by the estimates \eqref{L2 estimate}, \eqref{Hk estimate 1} and \eqref{Hk estimate 2},
\begin{align}\label{H4 estimate}
\frac{d}{dt}\|f\|_{H^4}(t)\le C(\|f\|_{H^4}^5+1).
\end{align}
Applying Gr\"{o}nwall's inequality on \eqref{H4 estimate}, for each $\alpha\in[0,1)$, there exists a positive time $T=T(\|f_0\|_{H^4},\alpha)$ and a positive constant $M=M(T,\alpha,\|f_0\|_{H^4})$ such that
\begin{align}\label{bound on H4 norm}
\sup_{0\le t\le T}\|f\|_{H^4}(t)\le M,
\end{align}
hence the local existence result for \eqref{contour eqn} follows by standard energy argument.

To prove the uniqueness of solution, let $f_1(x_1,x_2,t)$ and $f_2(x_1,x_2,t)$ be two solutions of \eqref{contour eqn} with $f_1(x_1,x_2,0)=f_2(x_1,x_2,0)=f_0((x_1,x_2)$. Using the bound \eqref{bound on H4 norm}, for $i\in\{1,2\}$,
\begin{align}\label{bound on f1 f2}
\sup_{0\le t\le T}\|f_i\|_{H^4}(t)\le M.
\end{align}

Define the difference by $\tilde f=f_1-f_2$. Then $\tilde f$ satisfies
\begin{align*}
\frac{d}{dt}\|\tilde f\|^2_{L^2}(t)=\mathbf{I}_{6}+\mathbf{I}_{7}+\mathbf{I}_{8},
\end{align*}
where
\begin{align*}
\mathbf{I}_{6}&=\int_{\R^2}\tilde f(x)\nabla \tilde f(x)\cdot \PV\int_{\R^2}y\left[ |y|^2 + (f_1(x) - f_1(x - y))^2 \right]^{-\frac{3 +\alpha}{2}}dydx,\\
\mathbf{I}_{7}&=-\int_{\R^2}\tilde f(x)\PV\int_{\R^2}\nabla \tilde f(y)\cdot(x-y)\left[ |x-y|^2 + (f_1(x) - f_1(y))^2 \right]^{-\frac{3 +\alpha}{2}},\\
\mathbf{I}_{8}&=\int_{\R^2}\tilde f(x)\PV\int_{\R^2}(\nabla f_2(x) - \nabla f_2(x-y))\cdot y N_{\alpha}(x,y)dydx,
\end{align*}
with
\begin{align*}
N_{\alpha}(x,y) = \left[ |y|^2 + (f_1(x) - f_1(x - y))^2 \right]^{-\frac{3 +\alpha}{2}} - \left[ |y|^2 + (f_2(x) - f_2(x - y))^2 \right]^{-\frac{3 +\alpha}{2}}.
\end{align*}
Integrating by parts and using the bound \eqref{bound on f1 f2}, we readily have
\begin{align*}
\mathbf{I}_{6} + \mathbf{I}_{8} \le C(M)\|\tilde f\|^2_{L^2},
\end{align*}
where $C(M)$ is a positive constant which depends on $M$. For the term $\mathbf{I}_{7}$, we estimate it as follows.
\begin{align*}
\mathbf{I}_{7}&=-\PV\int_{\R^2}\int_{\R^2}\tilde f(x)(\tilde f(x) - \tilde f(y))\left[ |x-y|^2 + (f_1(x) - f_1(y))^2 \right]^{-\frac{3 +\alpha}{2}}dydx\\
&\qquad+\PV\int_{\R^2}\int_{\R^2}\tilde f(x)(\tilde f(x) - \tilde f(y))\\
&\qquad\qquad\times\frac{(3+\alpha)(f_1(x) - f_2(y))(f_1(x) - f_1(y) - \nabla f_1(x)(x-y))}{\left[ |x-y|^2 + (f_1(x) - f_1(y))^2 \right]^{-\frac{5 +\alpha}{2}}}dydx\\
&\le C(M)\|\tilde f\|^2_{L^2}.
\end{align*}
Hence we obtain
\begin{align*}
\frac{d}{dt}\|\tilde f\|^2_{L^2}(t)\le C(M)\|\tilde f\|^2_{L^2}(t),
\end{align*}
and the uniqueness follows by applying Gr\"{o}wall's inequality.
\end{proof}
\begin{rem}\label{uniform bound remark}
In view of the proof of Theorem~\ref{local existence thm}, for each $\bar{\alpha}\in[0,1)$, there exists a positive constant $C(\bar{\alpha})$ such that for all $\alpha\in[0,\bar{\alpha})$, if $f_\alpha$ is a local-in-time solution to \eqref{contour eqn} with the same initial data $f_0$, then we have
\begin{align*}
\frac{d}{dt}\|f_\alpha\|_{H^4}(t)\le C(\bar{\alpha})(\|f_\alpha\|_{H^4}^5+1).
\end{align*}
Hence there exists a positive time $\bar{T}=\bar{T}(\|f_0\|_{H^4},\bar\alpha)$ and a positive constant $\bar{M}=\bar{M}(\bar{T},\bar{\alpha},\|f_0\|_{H^4})$ such that $f_\alpha$ are solutions to \eqref{contour eqn} for all $\alpha\in[0,\bar{\alpha})$ on $[0,T]$ satisfying
\begin{align}\label{uniform bound on H4 norm}
\sup_{0\le \alpha<\bar{\alpha}}\sup_{0\le t\le \bar{T}}\|f_\alpha\|_{H^4}(t)\le \bar{M}.
\end{align}
The uniform bound \eqref{uniform bound on H4 norm} will be useful for proving convergence of solutions as $\alpha\to0^+$ later in Section~\ref{convergence for alpha section}.
\end{rem}


\section{Global-in-time existence with bounded data}\label{global-in-time existence sec}

In this section, we obtain global-in-time existence for classical solutions of the contour equation \eqref{contour eqn}. We make use of the norm given by
\begin{equation}\label{def of fourier norm}
\left\|f\right\|_{s} \triangleq \int_{\mathbb{R}^2} d\xi |\xi|^{s} |\hat{f}(\xi)| \,  , \quad s \geq 1.
\end{equation}
The results are summarised in the following theorem:

\begin{thm}\label{global existence thm}
Let $\alpha\in[0,\frac{1}{2})$ be given. Assume that $f_0\in H^k(\R^2)$ for some $k\ge4$ and $\|f_0\|_{1}<k_0(\alpha)$, where $k_0(\alpha)$ is a constant that satisfies
\begin{align}\label{condition on k0}
2 C(\alpha) \sum_{n=1}^{\infty}   \left( 2n+1 \right)^{2}\frac{\Gamma\left(\frac{3 + \alpha}{2} + n\right)}{\Gamma\left(\frac{3 + \alpha}{2}\right)} \cdot \frac{k_0(\alpha)^{2n}}{n!}<1,
\end{align}
with $C(\alpha)$ being some positive constant. Then there exists a unique solution $f$ of \eqref{contour eqn} with initial data $f_0$ such that $f\in C([0,T]; H^k(\R^2))$ for any $T>0$.
\end{thm}

\begin{rem}
Note that the global-in-time existence result only holds for  $0 \le \alpha < \frac{1}{2}$. This is due to the fact that we have to ensure the set $(\alpha, 1 - \alpha )$ is non-empty. It is unknown whether we have global-in-time existence of solutions when $ \frac{1}{2} \le \alpha < 1 $.
\end{rem}

\begin{rem}
The existence of $k_0(\alpha)$ that appears in Theorem~\ref{global existence thm} is shown in Lemma~\ref{existence of k0 alpha lem} in Appendix~\ref{existence of k0 and estimates on principal values sec}.
\end{rem}

Before we give the proof of Theorem~\ref{global existence thm}, we prove some lemmas which are crucial for our analysis. {To begin with, we state and prove the following lemma for estimating a non-linear term $N_\alpha (f)$.

\begin{lemma}\label{estimate on non-linear N lem}
For $\alpha\in(0,\frac{1}{2})$, define
\begin{equation}\label{def of non-linear term N}
N_\alpha (f) = \frac{1}{2\pi} \int_{\mathbb{R}^2} \frac{y}{|y|^{2 + \alpha}} \nabla_x (D_y f(x)) \, R_\alpha (D_y f(x)) \, dy,
\end{equation}
where $R_\alpha$ and $D_y$ are given by
\begin{align*}
R_\alpha(z) &\triangleq 1 - \frac{1}{(1 + z^2)^{\frac{3 + \alpha}{2}}}
\end{align*}
and
\begin{align*}
D_y f (x)&\triangleq \left( f(x) - f(x - y)\right) |y|^{-1}.
\end{align*}
Then under the condition $\left\| f_{0} \right\|_{1} < k_{0}(\alpha)$, it holds 
\begin{equation}\label{estimate on N(f) term}
\int_{\mathbb{R}^{2}} |\xi| \, |\widehat{N_{\alpha}(f)}|(\xi) \, d\xi < \int_{\mathbb{R}^{2}} |\xi|^{2+\alpha} |\hat{f}(\xi)| \, d\xi.
\end{equation}
\end{lemma}

\begin{rem}
Lemma~\ref{estimate on non-linear N lem} shows that the linear dissipation dominates non-linear growth, which is essential for proving the forthcoming Lemma~\ref{f1 estimate lemma}.
\end{rem}

\begin{proof}
To prove \eqref{estimate on N(f) term}, we apply the method which is reminiscent of the one given in \cite{constantin2016muskat}. By changing variables, the term $N_\alpha(f)$ can be rewritten as
\begin{equation*}
  N_\alpha(f) = \frac{1}{4\pi} \int_{\mathbb{R}^2}  \nabla_x \left(D_y f(x)\right) \,  R_\alpha(D_{y} f(x)) - \nabla_x (D_{-y} f(x)) \, R_\alpha(D_{-y} f(x))  dy.
\end{equation*}
For \( |z| < 1 \), we have the following Taylor's series expansion for \( R_\alpha(z) \):
\begin{equation*}
R_{\alpha}(z) = - \sum_{n=1}^{\infty} (-1)^n a_n z^{2n},
\quad a_n = \frac{\Gamma\left(\frac{3 + \alpha}{2} + n\right)}{\Gamma\left(\frac{3 + \alpha}{2}\right)} \cdot \frac{1}{n!},
\end{equation*}
where $\Gamma(\cdot)$ is the Gamma function. Hence if $|D_{y} f(x)| \leq \left\|f\right\|_1$, then we have
\begin{equation*}
N_{\alpha}(f) = - \frac{1}{4 \pi} \sum_{n\geq1}^{\infty} (-1)^n a_n \int_{\mathbb{R}^2} \frac{y}{|y|^{2+\alpha}}  \left( \nabla_x (D_y  f) \, (D_y f)^{2n} - \nabla_{x} (D_{-y} f) (D_{-y} f)^{2n} \right) dy.
\end{equation*}
Recall the Fourier transform for $D_y  f$ and $ \nabla_x D_y  f$ are respectively given by
\begin{equation*}
\widehat{D_y  f} = \hat{f}{m}(\xi, y), \quad \widehat{\nabla_x D_y  f} = i \xi  \hat{f}(\xi){m}(\xi, y),
\end{equation*}
where the Fourier multiplier $m(\xi, y)$ is given by
\begin{align*}
m(\xi, y) = \frac{1 - e^{-i \xi \cdot y}}{|y|}.
\end{align*}
Hence we have
\begin{equation}\label{def of difference operator}
\reallywidehat{\nabla_x \left(D_y f\right) \, (D_y f)^{2n}} = \left((i \xi \hat{f}_m) * (\hat{f}_m) * \cdots * (\hat{f}_m)\right)(\xi, \alpha).
\end{equation}
By applying convolutions, we get
\begin{equation*}
  \begin{split}
\reallywidehat{{N}_\alpha(f)}(\xi) =&  \frac{-i}{4\pi}   \sum_{n\geq1} (-1)^n a_n \int_{\mathbb{R}^2} dy \int_{\mathbb{R}^2} d\xi_1 \cdots \int_{\mathbb{R}^2}  d\xi_{2n} \,\frac{y}{|y|^{2+\alpha}} \left(\xi-\xi_1\right)\\
&\times \hat{f}\left( \xi - \xi_1 \right)\left(\prod_{j=1}^{2n-1}    \hat{f}\left( \xi_{j} - \xi_{j+1} \right)\right)\hat{f}\left( \xi_{2n}  \right)  \left( M_{n}(y) - M_{n}(-y) \right),
 \end{split}
\end{equation*}
where $M_{n}$ is defined by
\begin{equation*}
M_{n}(y) \triangleq m \left( \xi - \xi_{1}, y \right) \left( \prod_{j=1}^{2 n - 1} m \left( \xi_{j} - \xi_{j+1}, y \right). \right) m \left( \xi_{2 n}, y \right).
\end{equation*}
Using  Fubini's theorem, we further get
\begin{equation*}
  \begin{split}
\reallywidehat{{N}_\alpha(f)}(\xi) =& \sum_{n \geq  1} a_n \int_{\mathbb{R}^2} d\xi_1\cdots \int_{\mathbb{R}^2} d\xi_{2n} \\
&\times \left( \xi - \xi_{1} \right) \hat{f}\left( \xi - \xi_{1} \right) \left(  \prod_{j=1}^{2 n - 1} \hat{f}\left( \xi_{j} - \xi_{j+1} \right)\right) \hat{f} \left( \xi_{2 n} \right) \cdot I_{n,\alpha}
 \end{split}
\end{equation*}
 with
 $$I_{n,\alpha} =  \frac{-i}{4\pi} (-1)^{n} \int_{\mathbb{R}^2} \frac{y}{|y|^{2 + \alpha}} \left( M_{n} (y) - M_{n}(-y) \right) dy.$$
Using the polar coordinates $(r,\theta)$ that $y = (r\cos \theta, r\sin \theta)$, $I_{n,\alpha}$ can be rewritten as
\begin{align*}
I_{n,\alpha}=   \frac{-i}{4\pi} (-1)^{n} \int_{-\pi}^{\pi} u \, d\theta \int_{0}^{+\infty}  \, dr \left( M_{n}(r, u) - M_{n}(r, -u) \right)
\end{align*}
with $m(\xi, r, u)$ and $M_{n}(r, u)$ being given by
\begin{equation*}
 m(\xi, r, u) =  \frac{\left(1-\exp ^{ir\xi \cdot u}\right)}{r|r|^{\alpha}}=\frac{i \xi \cdot u}{|r|^{\alpha}} \int_{0}^{1} e^{i r (s-1) \xi \cdot u} ds
\end{equation*}
and
\begin{align*}
M_{n}(r, u) &= m(\xi - \xi_1, r, u) \left( \prod_{j=1}^{2n-1} m(\xi_j - \xi_{j+1}, r, u) \right)  m\left(\xi_{2n}, r, u  \right)\\
&= (-1)^{n} \int_{0}^{1} ds_{1} \cdots \int_{0}^{1} ds_{2n} \left( \prod_{j=1}^{2n-1}  \left( \xi_{j} - \xi_{j+1} \right) \cdot u \right) \xi_{2n}\cdot u  \\
&\qquad\times \left(  \frac{e^{-i r \mathbf{S}_1\cdot u}}{r|r|^{\alpha}}   - \frac{e^{-i r \mathbf{S}_2\cdot u }}{r|r|^{\alpha}}  \right),
\end{align*}
where the terms $\mathbf{S}_1$ and $\mathbf{S}_2$ are defined by
\begin{align*}
\mathbf{S}_1 &= \sum_{j=1}^{2n-1} (s_{j} - 1) (\xi_{j} - \xi_{j+1}) + (s_{2n}-1) \xi_{2n},\\
\mathbf{S}_2 &= - (\xi - \xi_{1}) + \sum_{j=1}^{2n-1} (s_{j} - 1) (\xi_{j} - \xi_{j+1}) + (s_{2n}-1) \xi_{2n}.
\end{align*}
with $0 \leq s_j \leq 1$ and $ 1 \leq j \leq 2n$. Note that $m(\xi,-r,u)=-m(\xi,r,-u)$ and $-m(\xi,-r,-u)=m(\xi,r,u)$, upon changing of variables, it yields
\begin{align*}
I_{n,\alpha}=   \frac{-i}{8\pi} (-1)^{n} \int_{-\pi}^{\pi} u \, d\theta \int_{\R}  \, dr \left( M_{n}(r, u) - M_{n}(r, -u) \right).
\end{align*}
Since for $0 < \alpha < \frac{1}{2}$ and $\mathbf{S}\in\{\mathbf{S}_1, \mathbf{S}_2\}$,
\begin{align*}
|\mathbf{S}|^{\alpha} \leq \left[ \left| \xi - \xi_1 \right| + \left|\xi_1 - \xi_2\right| + \cdots + \left| \xi_{2n-1} - \xi_{2n} \right| +\left|\xi_{2n} \right| \right]^\alpha 
\end{align*}
and using Lemma~\ref{estimates on PV lem} in Appendix~\ref{existence of k0 and estimates on principal values sec}, we have
\begin{align*}
 \left|\PV\int_{0}^{+\infty}\frac{\exp(i r \mathbf{S})}{r^{1+\alpha}} \right|\leq C(\alpha)|\mathbf{S}|^{\alpha},
\end{align*}
for some positive constant $C(\alpha)$, hence we can estimate $I_{n,\alpha}$ as follows:
\begin{align*}
I_{n,\alpha} &= - \frac{i}{4 \pi} \int_{-\pi}^{\pi} u \, d\theta \quad \int_{0}^{1} ds_{1} \cdots \int_{0}^{1} ds_{2n} \left( \prod_{j=1}^{2n+1} (\xi_{j} - \xi_{j+1}) \right) \cdot u \\
&\quad \times \int_{0}^{+\infty} \left( \frac{ \exp(i r \mathbf{S}_1 \cdot u) }{ r^{1+\alpha} } - \frac{ \exp(i r \mathbf{S}_2 \cdot u) }{ r^{1+\alpha} } - \frac{ \exp(-i r \mathbf{S}_1 \cdot u) }{ r^{1+\alpha} } + \frac{ \exp(-i r \mathbf{S}_2 \cdot u) }{ r^{1+\alpha} } \right) dr\\
&\le2C(\alpha) \cdot \prod_{j=1}^{2n+1} \left| \xi_{j} - \xi_{j+1} \right| \left| \xi_{2n} \right| \\
&\qquad\qquad\qquad\qquad\qquad\times \left[ \left| \xi - \xi_1 \right| + \left|\xi_1 - \xi_2\right| + \cdots + \left| \xi_{2n-1} - \xi_{2n} \right| +\left|\xi_{2n} \right| \right]^\alpha.
\end{align*}
Therefore we have
\begin{equation}\nonumber
\begin{aligned}
\int_{\mathbb{R}^{2}} |\xi| \, |\widehat{N_{\alpha}(f)}|(\xi) &\leq 2 \left| {\Gamma}(-\alpha) \right|  \sum_{n\geq1}  a_n \int_{\mathbb{R}^2} d\xi \int_{\mathbb{R}^2} d\xi_1 \cdots \int_{\mathbb{R}^2}  d\xi_{2n}|\xi|\\
& \times \left|\xi-\xi_{1} \right| \left| \hat{f}(\xi  - \xi_{1} )  \right|   \prod_{j=1}^{2n+1} \left| \xi_{j} - \xi_{j+1} \right| \left| \hat{f}(\xi_{j} - \xi_{j+1} )  \right| \\
& \times\left| \xi_{2n} \right|  \left| \hat{f}(\xi_{2n})\right| \times   \left[ \left| \xi - \xi_1 \right| + \left|\xi_1 - \xi_2\right| + \cdots + \left| \xi_{2n-1} - \xi_{2n} \right| +\left|\xi_{2n} \right| \right]^\alpha.
\end{aligned}
\end{equation}
By the triangle inequality $|\xi| \leq |\xi - \xi_1| + \cdots + |\xi_{2n-1} - \xi_{2n}| + |\xi_{2n}|$ and the fact that
\begin{equation}\nonumber
\begin{aligned}
&\left[ |\xi - \xi_1| + \cdots + |\xi_{2n-1} - \xi_{2n}| + |\xi_{2n}| \right]^{1+\alpha}\\
& \quad \leq (2n+1)^{\alpha} \left[ \left|\xi - \xi_1\right|^{1+\alpha} + \left|\xi_1 - \xi_2\right|^{1+\alpha} + \cdots + \left|\xi_{2n-1} - \xi_{2n}\right|^{1+\alpha} + \left|\xi_{2n}\right|^{1+\alpha} \right],
\end{aligned}
\end{equation}
we further obtain
\begin{align}\label{estimate on N(f) term 2}
\int_{\mathbb{R}^2} |\xi| \, |\hat{N}_\alpha (f)(\xi)| \, d\xi &\leq 2 C(\alpha) \sum_{n=1}^{\infty}    \left( 2n+1 \right)^{1+\alpha}  a_{n} \left( \int_{\mathbb{R}^2} |\xi| \, |\hat{f}(\xi)| \right)^{2n}\notag\\
&\qquad \times \left( \int_{\mathbb{R}^2} |\xi|^{2+\alpha} |\hat{f}(\xi)| \right) \notag\\
&\leq \left( \int_{\mathbb{R}^2} |\xi|^{2+\alpha} |\hat{f}(\xi)| \right) \left[ 2 C(\alpha) \sum_{n=1}^{\infty}   \left( 2n+1 \right)^{1+\alpha} a_n  \left\| {f}\right\|_{1}^{2n} \right] \notag\\
&< \int_{\mathbb{R}^2} |\xi|^{2+\alpha} |\hat{f}(\xi)| \left[ 2 C(\alpha) \sum_{n=1}^{\infty}   \left( 2n+1 \right)^{2}a_n  \left\| {f}\right\|_{1}^{2n} \right].
\end{align}
To bound the summation term on the right side of \eqref{estimate on N(f) term 2}, for $|z|<1$, we recall that $$1-(1+z^{2})^{-\frac{3+\alpha}{2}} =- \sum_{n=1}^{\infty}  (-1)^n a_n  z^{2n}.$$
Multiply the above by $ z $, differentiate with respect to $z$ and replace $z$ by $iz$,
\begin{equation*}
\frac{1+(2+\alpha)z^2}{(1-z^2)^{\frac{5+\alpha}{2}}} -1= \sum_{n=1}^{\infty}    a_{n} (2n + 1)z^{2n}.
\end{equation*}
By repeating the process once more, we obtain
\begin{equation}\label{power series for an}
(1 - z^2)^{-\frac{7+\alpha}{2}} \left[ (10 + 4\alpha) z^2 + (2 + \alpha)^{2} z^{4} \right] - 1 = \sum_{n=1}^{\infty}  a_n (2n + 1)^{2} z^{2n}.
\end{equation}
By applying the identity \eqref{power series for an}, for $\|f\|_{1}<1$, we have
\begin{equation*}
\sum_{n=1}^{\infty}  (2n+1)^2 a_n \left\| f  \right\|_{1}^{2n} = \left( 1 - \left\| f  \right\|_{1}^{2} \right)^{-\frac{7+\alpha}{2}} \left[ (10 +4\alpha) \left\| f\right\|_{1}^2 + (2 + \alpha)^2 \left\| f \right\|_1^4 \right]-1.
\end{equation*}
Given $\alpha \in (0,\frac{1}{2})$, by Lemma~\ref{existence of k0 alpha lem} in Appendix~\ref{existence of k0 and estimates on principal values sec}, choosing a constant $k_0(\alpha) \in (0,1)$ such that for all $0 \leq z < k_{0}(\alpha),$ we always have
\begin{equation*}
\left[ \frac{(10 + 4 \alpha) z^2 + (2 + \alpha) z^4}{(1 - z^2)^{\frac{7 + \alpha}{2}}} -1\right] < \frac{1}{2C(\alpha)},
\end{equation*}
then it implies
\begin{equation*}
2C(\alpha) \sum_{n=1}^{\infty}  (2n+1)^2 a_n \left\|f \right\|_{1}^{2n} < 1,
\end{equation*}
wherever $\left\|f \right\|_{1} < k_0(\alpha)$, and hence the bound \eqref{estimate on N(f) term} follows.
\end{proof}

Using Lemma~\ref{estimate on non-linear N lem}, we state and prove the following lemma for estimating $\|f\|_{1}$.

\begin{lemma}\label{f1 estimate lemma}
Let $\alpha\in[0,\frac{1}{2})$ be given. Assume that $f_0\in H^k(\R^2)$ for some $k\ge4$ and $\|f_0\|_{s}<k_0(\alpha)$, where $k_0(\alpha)$ is a constant that satisfies \eqref{condition on k0}. If $f$ is a classical solution of \eqref{contour eqn} defined on $[0,T]$ with initial data $f_0$, then for any $t\in[0,T]$, we have
\begin{align}\label{bound on f in fourier norm}
\|f\|_{1}(t)\le\|f_0\|_{1}.
\end{align}
\end{lemma}

\begin{proof}
The case for $\alpha=0$ was proved in \cite{constantin2016muskat}, hence it suffices to consider the case when $\alpha\in(0,\frac{1}{2})$. First of all, we rewrite the contour equation \eqref{contour eqn} to obtain
\begin{equation*}
   \frac{\partial}{\partial t} f(x,t) = - \Lambda^{1+\alpha} (f) - N_\alpha(f), 
\end{equation*}
where $N_\alpha (f)$ satisfies \eqref{def of non-linear term N}. Recalling the definition \eqref{def of fourier norm} of \( \left\| f \right\|_{1} \) and the contour equation \eqref{contour eqn}, we have
\begin{equation}\nonumber
\begin{split}
\frac{d}{dt} \left\| f\right\|_{1}(t) &= \int_{\mathbb{R}^{2}}|\xi| \left(|\partial_t\hat{f}(\xi) \overline{\hat{f}(\xi)}|  \, + \hat{f}(\xi) \partial_t\hat{f}(\xi)\right) \, / \, 2 |\hat{f}(\xi)| d\xi  \\
&\leq - \int_{\mathbb{R}^{2}} |(\xi)|^{2+\alpha} |\hat{f}(\xi)| \, d\xi + \int_{\mathbb{R}^{2}} |\xi| \, |\widehat{N_{\alpha}(f)}|(\xi) \, d\xi.
\end{split}
\end{equation}
Therefore if the condition $\left\|f\right\|_{1} < k_{0}(\alpha)$ holds, using the bound \eqref{estimate on N(f) term}, we have
\begin{align*}
\frac{d}{dt} \left\| f\right\|_{1}(t) &\leq - \int_{\mathbb{R}^{2}} |(\xi)|^{2+\alpha} |\hat{f}(\xi)| \, d\xi + \int_{\mathbb{R}^{2}} |\xi| \, |\widehat{N_{\alpha}(f)}|(\xi) \, d\xi < 0,
\end{align*}
which gives
\begin{equation*}
\left\|f(t)\right\|_{1}(t) \leq \left\|f_0\right\| < k_{0}(\alpha)
\end{equation*}
as desired.
\end{proof}
}
Next, we are going to obtain some higher order estimates on $f$ with the help of \eqref{bound on f in fourier norm}. The results are listed in Lemma~\ref{f1+ estimate lemma} and Lemma~\ref{H3+ estimate lemma}.
\begin{lemma}\label{f1+ estimate lemma}
Assume that the hypotheses and notations of Lemma~\ref{f1 estimate lemma} are in force. Then for any $\delta\in(0,1-\alpha)$ and for any $t\in[0,T]$, we have
\begin{align}\label{bound on f in higher fourier norm 1}
\|f\|_{1+\delta}(t) + \mu(\alpha)\int_0^t\|f\|_{2+\delta+\alpha}(s)ds\le\|f_0\|_{1+\delta}
\end{align}
and
\begin{align}\label{bound on f in higher fourier norm 2}
\|f\|_{2+\delta}(t)\le\|f_0\|_{2+\delta}
\end{align}
for some $\mu(\alpha)>0$.
\end{lemma}
\begin{proof}
Recall that by Sobolev inequality, for any $\delta\in(0,1)$, $\|f_0\|_{1+\delta}\le C\|f_0\|_{H^3}$. By the definition \eqref{def of fourier norm},
\begin{equation*}
\frac{d}{dt} \Lambda^{1+\delta}f(x,t) = - \Lambda^{2+\delta + \alpha} (f) - \Lambda^{1+\delta}N_\alpha (f),
 \end{equation*}
and hence
 \begin{equation*}
\frac{d}{dt} \left\|f\right\|_{ 1+\delta   } (t) = - \int_{\mathbb{R}^2} |\xi|^{2+\delta+\alpha} f \, |\hat{f}(\xi)| \, d\xi
+ \int_{\mathbb{R}^2} |\xi|^{1+\delta} \hat{N}_{\alpha} \, | \xi)| \, d\xi.
\end{equation*}
Using the same method as given in the proof of Lemma~\ref{f1 estimate lemma}, for any $0<\delta<1-\alpha$, we have
\begin{align*}
\int_{\mathbb{R}^2} |\xi|^{1+\delta } |\hat{N_{\alpha}}|(\xi)\, d&\xi \leq 2C(\alpha) \sum_{n=1}^{\infty}  (2n+1)^{1+\delta+\alpha} a_n \left( \int_{ \mathbb{R}^2}   |\xi| \hat{f}(\xi)| \,\right)^{2n} \\
&\qquad\qquad\qquad\qquad\qquad\qquad\qquad\qquad\qquad\times \left( \int_{ \mathbb{R}^2}   |\xi|^{2+\delta+\alpha} |\hat{f}(\xi)|  \right)\\
&\leq \left(\int_{ \mathbb{R}^2}   |\xi|^{2+\delta+\alpha} |\hat{f}|  \right)  \left[ \, 2{|\Gamma {(- \alpha)}|} \sum_{n=1}^{\infty} (2n+1)^2 { a_n } \left\|f\right\|_1^{2n} \right].
\end{align*}
For each $\alpha\in(0,\frac{1}{2})$, we define $\mu(\alpha)\in(0,1)$ by
\begin{align*}
\mu(\alpha)=1-\left[ \, 2C(\alpha) \sum_{n=1}^{\infty} (2n+1)^2 { a_n } k_0(\alpha)^{2n} \right],
\end{align*}
then by the bound \eqref{condition on k0}, it implies $\mu(\alpha)>0$, and hence we have
\begin{equation*}
\frac{d}{dt} \left\|f\right\|_{ {1+\delta}}(t) + \mu(\alpha)\int_{ \mathbb{R}^2}   |\xi|^{2+\delta+\alpha} |\hat{f}|   \, d\xi\leq 0,
\end{equation*}
Upon integrating over $t$, the bound \eqref{bound on f in higher fourier norm 1} follows. By repeating the same argument, we can show that the bound \eqref{bound on f in higher fourier norm 2} holds as well.
\end{proof}
\begin{lemma}\label{H3+ estimate lemma}
Assume that the hypotheses and notations of Lemma~\ref{f1 estimate lemma} are in force. Then for any $\delta\in(\alpha,1-\alpha)$, we have
\begin{align}\label{bound on f in higher Hk norm}
\sup_{0\le t\le T}\|f\|_{H^{k}}(t)\le \bar{C}\|f_0\|_{H^{k}},
\end{align}
where $\bar{C}$ is a positive constant that depends on $\|f_0\|_{1+\delta}$, $\alpha$ and $k_0(\alpha)$ but is independent of $T$.
\end{lemma}
\begin{proof}
Recall that for $\delta\in(0,1)$, the homogeneous H\"{o}lder norm $\left|\,\cdot\,\right|_{C^{\delta}}$ is given by
 \begin{equation*}
\left|g\right|_{C^{\delta}} \triangleq \max_{x \neq y} \frac{|g(x) - g(y)|}{|x - y|^\delta}.
 \end{equation*}
To estimate the $H^k$-norm of $f$, we perform similar analysis as in the proof of Theorem~\ref{local existence thm} with some modifications. For example, in order to bound the terms
\begin{align*}
\mathbf{H}_{1} = \int_{\mathbb{R}^2} \partial_{x_1}^k f(x) \int_{\mathbb{R}^2} \frac{\nabla_{x} D_{y} f(x) \cdot y \, ( D_{y} f(x))^k (\partial_{x_1} D_{y} f(x))^k}{|y|^{2+\alpha} \left[ 1 + (D_{y} f(x))^2 \right]^{\frac{3+\alpha}{2}+k}} \, dy \, dx
\end{align*}
and
\begin{align*}
\mathbf{H}_{2} = \int_{\mathbb{R}^2} \partial_{x_1}^k f(x) \int_{\mathbb{R}^2} \frac{\nabla_{x} D_{y} f(x) \cdot y \, (  D_{y} f(x)) (\partial_{x_1} D_{y} f(x))^k}{|y|^{2+\alpha} \left[ 1 + (D_{y} f(x))^2 \right]^{\frac{3+\alpha}{2}+k-1}} \, dy \, dx,
\end{align*}
we have
\begin{align*}
\mathbf{H}_{1} + \mathbf{H}_{2} &\le 2 \int_{\mathbb{R}^2} |\partial_{x_1}^k f(x)| \int_{\mathbb{R}^2} \frac{|\nabla_x D_y f(x)|}{|y|^{1+\alpha}} \left|\partial_{x_1} D_y f(x)\right|^k dydx\\
&\leq C\left\|\partial_{x_1}^k f\right\|_{L^2}\left\|\nabla f\right\|_{L^2} \left\|\nabla f \right\|^{k-2}_{L^{\infty}}\left|\nabla f\right|_{C^{\delta}} \left\|\nabla^2 f\right\|_{L^{\infty}} \int_{|y|>1} \frac{|y|^{2\delta}}{|y|^{4+2\alpha}} \, dy \quad \\
&\quad + C\left\|\partial_{x_1}^k f\right\|_{L^2}\left\|\nabla^2 f\right\|_{L^4}^2\left|\nabla^2 f\right|_{C^{\delta}} \left\|\nabla^2 f\right\|^{k-2}_{L^{\infty}}\int_{|y|<1} \frac{|y|^{2\delta}}{|y|^{2+2\alpha}}dy.
\end{align*}
Since $0 < \alpha < \frac{1}{2}$ and $\delta\in(\alpha,1-\alpha)$, the two integrals appeared on the right side of the above inequality remain finite. And by the interpolation inequality that $\|\nabla^2 f\|_{L^4}^2\le \|\nabla f\|_{L^\infty}\|\nabla^3 f\|_{L^2}$, we further obtain
\begin{align*}
\mathbf{H}_{1} + \mathbf{H}_{2}\le &C\left\|\partial_{x_1}^k f\right\|_{L^2}\left\|\nabla f\right\|_{L^2} \left\|\nabla f \right\|^{k-2}_{L^{\infty}}\left|\nabla f\right|_{C^{\delta}} \left\|\nabla^2 f\right\|_{L^{\infty}}\\
&\qquad\qquad+C\left\|\partial_{x_1}^k f\right\|_{L^2}\|\nabla f\|_{L^\infty}\|\nabla^3 f\|_{L^2}\left|\nabla^2 f\right|_{C^{\delta}} \left\|\nabla^2 f\right\|^{k-2}_{L^{\infty}}.
\end{align*}
By performing similar estimates on $\|\partial_{x_2}^3 f\|_{L^2}$, there exists some polynomial functions $p_1(\cdot)$ and $p_2(\cdot)$ whose degrees depend on $k$ such that
\begin{align*}
\frac{d}{dt}\|f\|^2_{H^k}\le p_1(\|\nabla f\|_{L^\infty})p_2(\|\nabla^2 f\|_{L^\infty})(|\nabla f|_{C^\delta}+|\nabla^2 f|_{C^\delta})\|f\|_{H^k}^2.
\end{align*}
Using Fourier transform, for $i\in\N$, we have
\begin{align*}
\|\nabla^{i} f\|_{L^\infty}\le \|f\|_{i}
\end{align*}
and
\begin{align*}
|\nabla^{i} f|_{C^\delta}\le \|f\|_{i+\delta},
\end{align*}
using the bounds \eqref{bound on f in higher fourier norm 1}-\eqref{bound on f in higher fourier norm 2}, we obtain
\begin{align*}
\frac{d}{dt}\|f\|^2_{H^k}&\le p_1(\|f_0\|_{1+\delta})p_2(\|f_0\|_{2+\delta})(\|f\|_{1+\delta}+\|f\|_{2+\delta})\|f\|_{H^k}^2\notag\\
&\le 2p_1(\|f_0\|_{1+\delta})p_2(\|f_0\|_{2+\delta})\|f\|_{2+\delta}\|f\|_{H^k}^2,
\end{align*}
where the last inequality follows since $\|f\|_{1+\delta}\le\|f\|_{2+\delta}$. Applying Gr\"{o}nwall's inequality and using the bound \eqref{bound on f in higher fourier norm 1}, we have
\begin{align*}
\|f\|_{H^k}(t)\le\|f_0\|_{H^k}\exp\Big(\frac{Cp_1(\|f_0\|_{1+\delta})p_2(\|f_0\|_{2+\delta})\|f_0\|_{1+\alpha}}{\mu(\alpha)}\Big),
\end{align*}
and \eqref{bound on f in higher Hk norm} follows by choosing $\bar{C}=\exp\Big(\frac{Cp_1(\|f_0\|_{1+\delta})p_2(\|f_0\|_{2+\delta})\|f_0\|_{1+\alpha}}{\mu(\alpha)}\Big)$.
\end{proof}
\begin{proof}[Proof of Theorem~\ref{global existence thm}]
Given $k\ge4$ and $\alpha\in[0,\frac{1}{2})$, let $f_0\in H^k$ such that $\|f_0\|_{1}<k_0(\alpha)$ with \eqref{condition on k0} holds. By Theorem~\ref{local existence thm}, there exists a time $T>0$ such that the contour equation \eqref{contour eqn} possess a unique solution in $C^1([0,T];K^k(\R^2))$ with $f(x,0)=f_0$. Using the bound \eqref{bound on f in higher Hk norm}, the local-in-time solution can then be continued in $H^k$ for all time provided that $\|f_0\|_{1}$ is initially smaller than $k(\alpha)$ given by \eqref{condition on k0}.
\end{proof}


\section{Maximum principle and asymptotic behaviour of solutions}\label{max principle and long time section}

In this section, we prove some results regarding the maximum principle for the interface $f$ and further address the asymptotic behaviour for $f$. The results are given in the following subsections:

\subsection{Decay of $L^\infty$-norm of $f$}

In this subsection, we obtain the maximum principle for the $L^\infty$-norm for $f$ and show that $\|f\|_{L^\infty}$ decays in time $t$.

We begin with the following lemma about an ODE problem involving hypergeometric function.
\begin{lem}\label{ode lemma}
Let $\alpha\in[0,1)$. Consider the ODE problem that
\begin{align}
\label{ode problem} 
\left\{ \begin{array}{l}
g(0)=0, \\
{(1+z^2)g'(z) - (3+\alpha)z \cdot g(z) = (3+\alpha)z^2,\qquad z\in\R.}
\end{array}\right.
\end{align}
Then \eqref{ode problem} has a smooth solution $g:\R\to\R$ which is given by
\begin{align}\label{explicit form for ode solution}
g(z)=\frac{(3+\alpha)z^3(1+z^2)^\frac{3+\alpha}{2}}{3}{}_2F_1 \left( \frac{3}{2}, \frac{\alpha + 5}{2}; \frac{5}{2}; - z^{2} \right),
\end{align}
where ${}_2F_1(\,\cdot\,,\,\cdot\,;\,\cdot\,;\,\cdot\,)$ is the ordinary hypergeometric function.
\end{lem}
\begin{proof}
It can be verified by direct computation and we omit the details here.
\end{proof}

The following theorem gives the maximum principle of $\|f\|_{L^\infty}$.

\begin{thm}\label{max principle thm}
Let $f_0\in H^k(\R^2)$ for $k\ge4$. Then for any $\alpha\in[0,1)$, if $f$ is the unique solution to \eqref{contour eqn} on $[0,T]$, the following bound holds for all $t\in[0,T]$:
\begin{align}\label{max principle for f}
\|f\|_{L^\infty}(t)\le\|f_0\|_{L^\infty}.
\end{align}
\end{thm}

\begin{proof}
The case for $\alpha=0$ was proved in \cite{cordoba2009maximum}, and hence we can assume $\alpha\in(0,1)$. Using Theorem~\ref{local existence thm} and Sobolev inequality, there exists $T>0$ and a unique solution $f(x,t)$ defined on $\R^2\times[0,T]$ such that $f\in C^1([0,T]\times\R^2)$. 

Since $f(\cdot,t)\in H^k$ with $k\ge4$, by the Riemann-Lebesgue lemma, we have $f(x,t)\to0$ as $|x|\to\infty$, and hence there exists a point $x_*\in\R^2$ such that $f(x,t)$ reaches its maximum at $x=x_*$.

We define $\Psi(t)=f(x_*,t)$ and assume that $\Psi(t)>0$. By the Rademacher theorem, the function $\Psi(t)$ is differentiable almost everywhere on $t$ with $\frac{\partial\Psi}{\partial t}(t)$ given by
\begin{align*}
\frac{\partial\Psi}{\partial t}(t)=\frac{\partial f}{\partial t}(x_*,t).
\end{align*}
Since $\nabla f(x_*,t) = 0$, we further obtain
\begin{equation*}
\frac{\partial\Psi}{\partial t}(t) = \PV \int_{\mathbb{R}^2} \frac{- \nabla f(y,t) \cdot (x_* - y)}{\left[ |x_* - y|^2 + (f(x_*,t) - f(y,t))^2 \right]^{\frac{3+ \alpha}{2}}} \, dy.
\end{equation*}
Upon integrating by parts,
\begin{align*}
\frac{\partial\Psi}{\partial t}(t) &= - \PV \int_{\mathbb{R}^2} \left( f{(x_*,t)} - f{(y,t)} \right) \nabla_x\cdot\left(\frac{x_*- y}{|x_* - y|^{3+\alpha}} \right) \left[ 1 + \left( \frac{f{(x_*,t)} - f{(y,t)} }{|x_* - y|}\right)^2 \right]^{-\frac{3 + \alpha}{2}}\\
&\quad - \PV \int_{\mathbb{R}^2} \frac{\left( f{(x_*,t)} - f{(y,t)} \right)}{|x_* - y|}\frac{x_* - y}{|x_* - y|^{2+\alpha}} \cdot\nabla_y \left[ 1 + \left( \frac{f{(x_*,t)} - f{(y,t)} }{|x_* - y|}\right)^2 \right]^{-\frac{3 + \alpha}{2}}\\
&\triangleq \mathbf{D}_1+\mathbf{D}_2.
\end{align*}
We estimate $\mathbf{D}_1$ and $\mathbf{D}_2$ as follows. First of all, for the term $\mathbf{D}_1$, since $\Psi(t)\ge f(y,t)$ for all $y\in\R^2$, we readily have
\begin{align*}
\mathbf{D}_{1} = - \PV \int_{\mathbb{R}^2} \frac{\Psi(t) - f(y,t)}{\left[ \left|x_* - y\right|^{2} + \left(\Psi(t) - f(y,t)\right)^{2}\right]^{\frac{3 + \alpha}{2}}}dy\le 0.
\end{align*}
{For the term $\mathbf{D}_{2}$, using Lemma~\ref{ode lemma}, if we define $g(z)$ by \eqref{explicit form for ode solution}, then by choosing 
\begin{align*}
H(z)=\frac{g(z)}{(1+z^2)^\frac{3+\alpha}{2}},
\end{align*}
we have
\begin{align*}
H'(z)&=g'(z)(1+z^2)^{-\frac{3+\alpha}{2}}-\frac{3+\alpha}{2}g(z)(2z)(1+z^2)^{-\frac{5+\alpha}{2}}\\
&=(3+\alpha)z^2(1+z^2)^{-\frac{5+\alpha}{2}},
\end{align*}
and hence $\mathbf{D}_{2}$ can be expressed by
\begin{align*}
\mathbf{D}_{2}= \PV \frac{1}{\alpha} \int_{\mathbb{R}^2} \nabla_y \left( |x_* - y|^{-\alpha} \right) \cdot \nabla_y H \left( \frac{f{(x_*,t)} - f{(y,t)}}{|x_* - y|} \right) dy.
\end{align*}
For each $0<r\le 1$, let $\mathcal{B}_r(x_{*})$ be the ball centred at $x_{*}$ with radius $r$, then we can further decompose $\mathbf{D}_{2}$ as follows.
\begin{align}\label{D2 def revised}
\mathbf{D}_{2}&=\frac{1}{\alpha} \int_{\mathcal{B}_r(x_{*})} \nabla_y \left( |x_* - y|^{-\alpha} \right) \cdot \nabla_y H \left( \frac{f{(x_*,t)} - f{(y,t)}}{|x_* - y|} \right) dy\notag\\
&\qquad+\frac{1}{\alpha} \int_{\mathbb{R}^2\setminus\mathcal{B}_r(x_{*})} \nabla_y \left( |x_* - y|^{-\alpha} \right) \cdot \nabla_y H \left( \frac{f{(x_*,t)} - f{(y,t)}}{|x_* - y|} \right) dy.
\end{align}
For the first integral, there exists $C(x_*,f,\alpha)>0$ independent of $r$ such that
\begin{align*}
\left|\frac{1}{\alpha} \int_{\mathcal{B}_r(x_{*})} \nabla_y \left( |x_* - y|^{-\alpha} \right) \cdot \nabla_y H \left( \frac{f{(x_*,t)} - f{(y,t)}}{|x_* - y|} \right) dy\right|\le C(x_*,f,\alpha)r^{1-\alpha},
\end{align*} 
while for the second integral, upon integrating by parts implies
\begin{align*}
&\frac{1}{\alpha} \int_{\mathbb{R}^2\setminus\mathcal{B}_r(x_{*})} \nabla_y \left( |x_* - y|^{-\alpha} \right) \cdot \nabla_y H \left( \frac{f{(x_*,t)} - f{(y,t)}}{|x_* - y|} \right) dy\\
&=-\frac{1}{\alpha} \int_{\mathbb{R}^2\setminus\mathcal{B}_r(x_{*})} \Delta_y \left( |x_* - y|^{-\alpha} \right) \cdot H \left( \frac{f{(x_*,t)} - f{(y,t)}}{|x_* - y|} \right) dy\\
&\qquad+\frac{1}{\alpha} \int_{\partial\mathcal{B}_r(x_{*})} \nabla_y \left( |x_* - y|^{-\alpha} \right) \cdot H \left( \frac{f{(x_*,t)} - f{(y,t)}}{|x_* - y|} \right) dS_y.
\end{align*}
Since $x_*$ is a spatial maximum of $f$, it implies $\nabla f(x_*) = 0$ and hence
\begin{align*}
\frac{f{(x_*,t)} - f{(y,t)}}{|x_* - y|} =\mathcal{O}(|x_* - y|).
\end{align*}
By the definition of $H(z)$, there exists $\tilde{C}(x_*,f,\alpha)>0$ independent of $r$ such that
\begin{align*}
\left|\frac{1}{\alpha} \int_{\partial\mathcal{B}_r(x_{*})} \nabla_y \left( |x_* - y|^{-\alpha} \right) \cdot H \left( \frac{f{(x_*,t)} - f{(y,t)}}{|x_* - y|} \right) dS_y\right|\le \tilde{C}(x_*,f,\alpha)r^{3-\alpha}.
\end{align*}
Therefore, by taking $r\to0$ in \eqref{D2 def revised}, we obtain
\begin{align*}
\mathbf{D}_{2}=-\frac{1}{\alpha} \int_{\mathbb{R}^2} \Delta_y \left( |x_* - y|^{-\alpha} \right) \cdot H \left( \frac{f{(x_*,t)} - f{(y,t)}}{|x_* - y|} \right) dy.
\end{align*}
We recall that $x_*$ is a global maximum of $f$, which implies $f(x_*) - f(y) \ge 0$ and consequently 
$H \left( \frac{f{(x_*,t)} - f{(y,t)}}{|x_* - y|} \right) \ge 0$ for all $y\in\R^2$. Together with the fact that $\Delta_y \left( |x_* - y|^{-\alpha} \right)\ge0$, it yields $\mathbf{D}_{2}\le0$.

Since both $\mathbf{D}_1 \le 0$ and $\mathbf{D}_{2}\le0$, we finally obtain $\frac{\partial \Psi}{\partial t}(t)\le 0$ for almost every $t$. Integrating over $t$, it implies $f(x_*,t)\le\|f_0\|_{L^\infty}$ for all $t>0$.}

For the case when $\Psi(t)\le 0$, following the similar argument as shown above, we can conclude that $\frac{\partial \Psi}{\partial t}(t)\ge 0$ for almost every $t$. Hence upon integrating over $t$, it implies $-f(x_*,t)\le\|f_0\|_{L^\infty}$ for all $t>0$. Combining the above cases, we prove that \eqref{max principle for f} holds.
\end{proof}

As a consequence of Theorem~\ref{max principle thm}, the following theorem gives the decay of $L^\infty$-norm of $f$.

\begin{thm}\label{decay of f thm}
Let $f_0\in H^k(\R^2)$ with $k\ge4$ and $\alpha\in[0,1)$. If $f_0$ satisfies
\begin{align}\label{positivity assumption on f0}
\text{$f_0(x)\le0$ or $f_0(x)\ge0$ for all $x\in\R^2$,}
\end{align}
and if $f$ is the unique solution to \eqref{contour eqn} on $[0,T]$, then for all $t\in[0,T]$, it satisfies
\begin{align}\label{decay on f in L infty norm}
\|f\|_{L^\infty}(t)\le\frac{\|f_0\|_{L^\infty}}{[1+C(\|f_0\|_{L^\infty},\|f_0\|_{L^1},\alpha)t]^\frac{2}{1+\alpha}},
\end{align}
where $C(\|f_0\|_{L^\infty},\|f_0\|_{L^1},\alpha)$ is a positive constant which depends on $\|f_0\|_{L^\infty}$, $\|f_0\|_{L^1}$ and $\alpha$.
\end{thm}

\begin{proof}
The inequality \eqref{decay on f in L infty norm} holds trivially for $\|f\|_{L^\infty}(t)=0$, and hence without loss of generality we assume $\|f\|_{L^\infty}(t)>0$. 

Consider the case for $f_0(x)\ge0$, and the case for $f_0(x)\le0$ can be proved by a similar argument. By Theorem~\ref{max principle thm},
\begin{equation*}
\frac{\partial\Psi}{\partial t}(t) = \PV \int_{\mathbb{R}^2} \frac{- \nabla f(y,t) \cdot (x_* - y)}{\left[ |x_* - y|^2 + (f(x_*,t) - f(y,t))^2 \right]^{\frac{3+ \alpha}{2}}} \, dy\ge0,
\end{equation*}
for almost every $t$. Hence if $f_0(x)\ge0$, then it implies $f(x,t)\ge0$. Also, by symmetry, 
\begin{align*}
\int_{\R^2}\frac{\partial f}{\partial t}(x,t)dx = \int_{\R^2} \PV \, \int_{\mathbb{R}^2} \frac{ \left( \nabla f(x,t) - \nabla f(x - y, t) \right) \cdot y }{ \left[ |y|^2 + (f(x,t) - f(x - y, t))^2 \right]^{\frac{3 + \alpha}{2}}} \, dy = 0,
\end{align*}
hence we have $\|f\|_{L^1}(t)=\|f_0\|_{L^1}$. By the fact that $\|f\|_{L^\infty}(t)=f(x_*,t)$, we compute
\begin{align*}
\frac{d}{dt}\|f\|_{L^\infty}(t)=-\mathbf{D}_3,
\end{align*}
where $\mathbf{D}_3$ is given by
\begin{align*}
\mathbf{D}_3=\PV \int_{\mathbb{R}^2} \frac{f(x_*,t) - f(y,t)}{\left[ |x_* - y|^2 + (f(x_*,t) - f(y,t))^2 \right]^{\frac{3+ \alpha}{2}}} \, dy.
\end{align*}
For each $r>0$, we define the ball $\mathcal{B}_r(x_*)=\{y:|x_*-y|\le r\}$ and the sets $V_1$, $V_2$ by
\begin{align*}
V_1 &= \left\{y\in \mathcal{B}_r(x_*): f(x_*,t) - f(y,t) \ge \frac{f(x_*,t)}{2}\right\},\\
V_2 &= \left\{y\in \mathcal{B}_r(x_*): f(x_*,t) - f(y,t) < \frac{f(x_*,t)}{2}\right\}.
\end{align*}
Since the integrand of $\mathbf{D}_3$ is non-negative, we can estimate $\mathbf{D}_3$ as follows.
\begin{align*}
\mathbf{D}_3 &\ge \int_{V_1} \frac{f(x_*,t) - f(y,t)}{\left[ |x_* - y|^2 + (f(x_*,t) - f(y,t))^2 \right]^{\frac{3+ \alpha}{2}}} \, dy\\
&\ge \frac{f(x_*,t)/2}{[r^2+4\|f_0\|^2_{L^\infty}]^\frac{3+\alpha}{2}}|V_1|=\frac{\|f\|_{L^\infty}(t)/2}{[r^2+4\|f_0\|^2_{L^\infty}]^\frac{3+\alpha}{2}}|V_1|.
\end{align*}
To estimate $|V_1|$, we notice that $|V_1|=\pi r^2 - |V_2|$ and since $f(y,t)\ge0$,
\begin{align*}
\|f_0\|_{L^1}&=\int_{\R^2} f(y,t)dy\\
&\ge\int_{V_2}f(y,t)dy\\
&\ge\frac{f(x_*,t)}{2}|V_2|=\frac{\|f\|_{L^\infty}(t)}{2}|V_2|.
\end{align*}
Hence by rearranging terms,
\begin{align*}
|V_1|\ge\pi r^2 - \frac{2\|f_0\|_{L^1}}{\|f\|_{L^\infty}(t)}.
\end{align*}
Therefore we have 
\begin{align*}
\mathbf{D}_3\ge \frac{\pi r^2\|f\|_{L^\infty}(t) - 2\|f_0\|_{L^1}}{2[r^2+4\|f_0\|^2_{L^\infty}]^\frac{3+\alpha}{2}}.
\end{align*}
We choose 
\begin{align*}
r=\left(\frac{2\|f_0\|_{L^1}/\pi+1}{\|f\|_{L^\infty}(t)}\right)^\frac{1}{2},
\end{align*}
then we obtain
\begin{align*}
\mathbf{D}_3\ge\frac{\pi}{2}\frac{\|f\|_{L^\infty}(t)^\frac{3+\alpha}{2}}{[1+2\|f_0\|_{L^1}/\pi+4\|f_0\|^3_{L^\infty}]^\frac{3+\alpha}{2}}.
\end{align*}
Choosing $\tilde C(\|f_0\|_{L^\infty},\|f_0\|_{L^1},\alpha)=\frac{\pi}{2[1+2\|f_0\|_{L^1}/\pi+4\|f_0\|^3_{L^\infty}]^\frac{3+\alpha}{2}}$, we thus conclude that
\begin{align*}
\frac{d}{dt}\|f\|_{L^\infty}(t)\le - \tilde C(\|f_0\|_{L^\infty},\|f_0\|_{L^1},\alpha)\|f\|_{L^\infty}(t)^\frac{3+\alpha}{2},
\end{align*}
and upon integrating by parts, the bound \eqref{decay on f in L infty norm} follows by taking the constant $C(\|f_0\|_{L^\infty},\|f_0\|_{L^1},\alpha)=\frac{1+\alpha}{2}\|f_0\|_{L^\infty}^\frac{1+\alpha}{2}\tilde C(\|f_0\|_{L^\infty},\|f_0\|_{L^1},\alpha)$.
\end{proof}

\subsection{Decay of $L^\infty$-norm of $\nabla f$}

In this subsection, we further obtain the maximum principle for the $L^\infty$-norm for $\nabla f$ and show that $\|\nabla f\|_{L^\infty}$ decays in time $t$ under a boundedness assumption on $\|\nabla f\|_{L^1}$.

We begin with the following maximum principle for $\nabla f$ under the assumption that $\nabla f_0$ is smaller than some constant depending on $\alpha$.

\begin{thm}\label{max principle nabla f thm}
Let $f_0\in H^k(\R^2)$ for $k\ge4$. Then for any $\alpha\in[0,1)$ and $\varepsilon\in(0,1+\alpha)$, if $f$ is the unique solution to \eqref{contour eqn} on $[0,T]$ with 
\begin{align}\label{boundedness assumption on nabla f0}
\|\nabla f_0\|_{L^\infty}<\sqrt{\frac{1+\alpha-\varepsilon}{5+\alpha+\varepsilon}},
\end{align}
the following bound holds for all $t\in[0,T]$:
\begin{align}\label{max principle for nabla f}
\|\nabla f\|_{L^\infty}(t)\le\|\nabla f_0\|_{L^\infty}.
\end{align}
\end{thm}

\begin{proof}
The method is reminiscent of the one given in \cite{constantin2016muskat} with some refinements. For $i\in\{1,2\}$, we apply $\partial_{x_i}$ on \eqref{contour eqn} to get
\begin{align*}
\frac{\partial}{\partial t}(\partial_{x_i}f(x,t))=\mathbf{Q}^{i}_1(x,t)+\mathbf{Q}^{i}_2(x,t)+\mathbf{Q}^{i}_3(x,t),
\end{align*}
where
\begin{align*}
\mathbf{Q}^{i}_1(x,t)&=\PV\int_{\R^2}\frac{\nabla\partial_{x_i}f(x,t)\cdot y}{[|y|^2+(f(x,t)-f(x-y,t))^2)]^\frac{3+\alpha}{2}}dy,\\
\mathbf{Q}^{i}_2(x,t)&=-\PV\int_{\R^2}\frac{\nabla\partial_{x_i}f(x-y,t)\cdot y}{[|y|^2+(f(x,t)-f(x-y,t))^2)]^\frac{3+\alpha}{2}}dy,\\
\mathbf{Q}^{i}_3(x,t)&=-\PV\int_{\R^2}\frac{\partial_{x_i}f(x,t) - \partial_{x_i}f(x-y,t)}{[|y|^2+(f(x,t)-f(x-y,t))^2)]^\frac{3+\alpha}{2}}E_{\alpha}(x,y,t)dy
\end{align*}
with
\begin{align*}
E_{\alpha}(x,y,t)=(3+\alpha)\left[\frac{(f(x,t)-f(x-y,t))(\nabla f(x,t)-\nabla f(x-y,t))\cdot y}{[|y|^2+(f(x,t)-f(x-y,t))^2)}\right].
\end{align*}
Integrating by parts, $\mathbf{Q}^{i}_2$ is given by
\begin{align*}
\mathbf{Q}^{i}_2&=\PV \int_{\R^2}\frac{2\partial_{x_i}f(x,t) - 2\partial_{x_i}f(x-y,t)}{[|y|^2+(f(x,t)-f(x-y,t))^2)]^\frac{3+\alpha}{2}}dy\\
&\qquad - \PV\int_{\R^2}\frac{2\partial_{x_i}f(x,t) - 2\partial_{x_i}f(x-y,t)}{[|y|^2+(f(x,t)-f(x-y,t))^2)]^\frac{3+\alpha}{2}}B_{\alpha}(x,y,t)dy
\end{align*}
with
\begin{align*}
B_{\alpha}(x,y,t)=(3+\alpha)\left[\frac{(f(x,t)-f(x-y,t))\nabla f(x-y,t)\cdot y + |y|^2}{[|y|^2+(f(x,t)-f(x-y,t))^2)}\right].
\end{align*}
Hence by adding $\mathbf{Q}^{i}_2$ and $\mathbf{Q}^{i}_3$,
\begin{align*}
\mathbf{Q}^{i}_2+\mathbf{Q}^{i}_3=-\PV\int_{\R^2}\frac{\partial_{x_i}f(x,t) - \partial_{x_i}f(x-y,t)}{[|y|^2+(f(x,t)-f(x-y,t))^2)]^\frac{3+\alpha}{2}}C_{\alpha}(x,y,t)dy
\end{align*}
where 
\begin{align*}
C_{\alpha}(x,y,t)&=E_{\alpha}(x,y,t)+B_{\alpha}(x,y,t)-2\\
&=1+\alpha+\frac{3+\alpha}{1+(D_y f(x)))^2}\times\left[(D_y f(x))\nabla f(x-y)\cdot\frac{y}{|y|}-(D_y f(x))^2\right]
\end{align*}
with $D_y(\cdot)$ being defined in \eqref{def of difference operator}. Define a function $\Phi(t)$ by
\begin{align}\label{def of phi}
\Phi(t)=\max_{x\in\R^2}\left\{\sum_{i=1}^2(\partial_{x_i}f(x,t))^2\right\},
\end{align}
then there exists $x_{\#}\in\R^2$ such that $\Phi(t)=\sum_{i=1}^2(\partial_{x_i}f(x_{\#},t))^2$. Therefore we have
\begin{align}\label{de for phi}
\frac{\partial\Phi}{\partial t}(t)=2\sum_{i=1}^2\partial_{x_i}f(\mathbf{Q}^{i}_1+\mathbf{Q}^{i}_2+\mathbf{Q}^{i}_3)(x_{\#},t).
\end{align}
Since $\sum_{i=1}^2(\partial_{x_i}f(x,t))^2$ attains its maximum at $x=x_{\#}$, dropping the variable $t$ for simplicity, we have 
\begin{align*}
\sum_{i=1}^2\partial_{x_i}f(x_{\#})\mathbf{Q}^{i}_1(x_\#)=\PV\int_{\R^2}\frac{\sum_{i=1}^2\partial_{x_i}f(x_{\#})\nabla\partial_{x_i}f(x_{\#})\cdot y}{[|y|^2+(f(x_{\#})-f(x_{\#}-y))^2)]^\frac{3+\alpha}{2}}dy=0.
\end{align*}
Also, by symmetry, we have
\begin{align*}
\PV\int_{\R^2}\frac{\sum_{i=1}^2\partial_{x_i}f(x_{\#})\partial_{x_i}f(x_{\#}-y)}{[|y|^2+(f(x_{\#})-f(x_{\#}-y))^2)]^\frac{3+\alpha}{2}}C_{\alpha}(x_{\#},y)=0,
\end{align*}
and hence the identity \eqref{de for phi} can be reduced to
\begin{align}\label{de for phi revised}
\frac{\partial\Phi}{\partial t}(t)=-\PV\int_{\R^2}\frac{\sum_{i=1}^2(\partial_{x_i}f(x_{\#}))^2}{[|y|^2+(f(x_{\#})-f(x_{\#}-y))^2)]^\frac{3+\alpha}{2}}C_{\alpha}(x_{\#},y).
\end{align}
Since for $\|\nabla f\|_{L^\infty}<\sqrt{\frac{1+\alpha-\varepsilon}{5+\alpha+\varepsilon}}$, we have
\begin{align}\label{bound on C(x,y)}
C_{\alpha}(x_{\#},y)\ge 1+\alpha-\frac{2(3+\alpha)\|\nabla f\|^2_{L^\infty}}{1+\|\nabla f\|^2_{L^\infty}}>\varepsilon,
\end{align}
it implies that the right side of \eqref{de for phi revised} is always non-positive. By a bootstrap argument, we can show that if $\Phi(0)<\frac{1+\alpha-\varepsilon}{5+\alpha+\varepsilon}$, then $\Phi(t)<\frac{1+\alpha-\varepsilon}{5+\alpha+\varepsilon}$ for all $t\ge0$. It completes the proof of \eqref{max principle for nabla f}.
\end{proof}
\begin{rem}
By taking $\alpha=0$ and $\varepsilon=\frac{2}{5}$, Theorem~\ref{max principle nabla f thm} gives the results of \cite[Theorem~4.1]{constantin2016muskat}.
\end{rem}

With the help of Theorem~\ref{max principle nabla f thm}, the following theorem gives a decay estimate on $\|\nabla f\|_{L^\infty}$ in time provided that $\|\nabla f\|_{L^1}$ is bounded.

 \begin{thm}\label{decay of nabla f thm}
Let $f_0\in H^k(\R^2)$ for $k\ge4$. Then for any $\alpha\in[0,1)$ and $\varepsilon\in(0,1+\alpha)$, if $f$ is the unique solution to \eqref{contour eqn} on $[0,T]$ with $f_0$ satisfies \eqref{boundedness assumption on nabla f0}, and if there exists $M_{T}>0$ such that 
\begin{align}\label{boundedness condition on nabla f in L1}
\sup_{t\in[0,T]}\|\nabla f\|_{L^1}(t)\le M_{T},
\end{align}
then the following bound holds for all $t\in[0,T]$:
\begin{align}\label{decay on nabla f in L infty norm}
\|\nabla f\|_{L^\infty}(t)\le\frac{\|\nabla f_0\|_{L^\infty}}{[1+C_{\#}(\|f_0\|_{L^\infty},\|\nabla f_0\|_{L^\infty},M_{T},\varepsilon,\alpha)t]^\frac{2}{1+\alpha}},
\end{align}
where $C_{\#}(\|f_0\|_{L^\infty},\|\nabla f_0\|_{L^\infty},M_{T},\varepsilon,\alpha)>0$ is a constant which depends on $\|f_0\|_{L^\infty}$, $\|\nabla f_0\|_{L^\infty}$, $M_{T}$, $\varepsilon$ and $\alpha$.
\end{thm}
\begin{proof}
Without loss of generality, we assume that $\|\nabla f\|_{L^\infty}(t)>0$. Recall from the proof of Theorem~\ref{max principle nabla f thm} that
\begin{align}\label{de for phi using Q4}
\frac{\partial\Phi}{\partial t}(t)=-\mathbf{Q}_4(x_{\#},t),
\end{align}
where $\Phi$ is defined in \eqref{def of phi} and $\mathbf{Q}_4$ is given by
\begin{align*}
\mathbf{Q}_4(x_{\#},t)=\PV\int_{\R^2}\frac{\sum_{i=1}^2\partial_{x_i}f(x_{\#},t)(\partial_{x_i}f(x_{\#},t)-\partial_{x_i}f(x_{\#}-y,t))}{[|y|^2+(f(x_{\#},t)-f(x_{\#}-y,t))^2)]^\frac{3+\alpha}{2}}C_{\alpha}(x_{\#},y,t).
\end{align*}
Dropping the variable $t$ for simplicity, we define the sets $W_1$, $W_2$ by
\begin{align*}
W_1 &= \left\{y\in \mathcal{B}_r(x_{\#}): \sum_{i=1}^2(\partial_{x_i}f(x_{\#})(\partial_{x_i}f(x_{\#}) - \partial_{x_i}f(x_{\#}-y)) \ge \frac{1}{2}\sum_{i=1}^2(\partial_{x_i}f(x_{\#}))^2\right\},\\
W_2 &= \left\{y\in \mathcal{B}_r(x_{\#}): \sum_{i=1}^2(\partial_{x_i}f(x_{\#})(\partial_{x_i}f(x_{\#}) - \partial_{x_i}f(x_{\#}-y)) < \frac{1}{2}\sum_{i=1}^2(\partial_{x_i}f(x_{\#}))^2\right\},
\end{align*}
where $\mathcal{B}_r(x_{\#})$ is the ball centred at $x_{\#}$ with radius $r$. Since the integrand of $\mathbf{Q}_4$ is non-negative, we apply the bounds \eqref{max principle for f} and \eqref{bound on C(x,y)} to get
\begin{align*}
\mathbf{Q}_4&\ge \int_{W_1}\frac{\sum_{i=1}^2\partial_{x_i}f(x_{\#})(\partial_{x_i}f(x_{\#})-\partial_{x_i}f(x_{\#}-y))}{[|y|^2+(f(x_{\#})-f(x_{\#}-y))^2)]^\frac{3+\alpha}{2}}C_{\alpha}(x_{\#},y)\\
&\ge\frac{\varepsilon\Phi(t)}{2[r^2+4\|f_0\|^2_{L^\infty}]^\frac{3+\alpha}{2}}|W_1|.
\end{align*}
On the other hand, $|W_1|=\pi r^2-|W_2|$ and using the bound \eqref{boundedness condition on nabla f in L1},
\begin{align*}
M_T\Phi(t)^\frac{1}{2}\ge\|\nabla f\|_{L^1}\|\nabla f\|_{L^\infty}&\ge\int_{W_2}\sum_{i=1}^2\partial_{x_i}f(x_{\#})\partial_{x_i}f(x_{\#}-y)\\
&\ge\int_{W_2}\left(\sum_{i=1}^2(\partial_{x_i}f(x_{\#}))^2-\frac{1}{2}\sum_{i=1}^2(\partial_{x_i}f(x_{\#}))^2\right)\\
&\ge\frac{|W_2|\Phi(t)}{2}.
\end{align*}
Hence we have
\begin{align*}
|W_1|\ge\frac{\pi r^2\Phi(t)^\frac{1}{2}-2M_T}{\Phi(t)^\frac{1}{2}},
\end{align*}
which further implies
\begin{align*}
\mathbf{Q}_4\ge\frac{\varepsilon\Phi(t)^\frac{1}{2}(\pi r^2\Phi(t)^\frac{1}{2}-2M_T)}{2[r^2+4\|f_0\|^2_{L^\infty}]^\frac{3+\alpha}{2}}.
\end{align*}
By choosing $r=\left(\frac{2M_T/\pi + 1}{\Phi^\frac{1}{2}(t)}\right)^\frac{1}{2}$ and applying the bound \eqref{max principle for nabla f} to $\Phi(t)^\frac{1}{2}$, we obtain
\begin{align}\label{lower bound on Q4}
\mathbf{Q}_4\ge\frac{\varepsilon\Phi(t)^\frac{5+\alpha}{4}}{2[1+2M_T/\pi+4\|f_0\|^2_{L^\infty}\|\nabla f_0\|_{L^\infty}]^\frac{3+\alpha}{2}}.
\end{align}
Using \eqref{lower bound on Q4} on \eqref{de for phi using Q4}, we conclude that
\begin{align*}
\frac{\partial\Phi}{\partial t}(t)\le -\frac{\varepsilon\Phi(t)^\frac{5+\alpha}{4}}{2[1+2M_T/\pi+4\|f_0\|^2_{L^\infty}\|\nabla f_0\|_{L^\infty}]^\frac{3+\alpha}{2}},
\end{align*}
and by integrating the above inequality over time $t$, there exists a positive constant $C_{\#}(\|f_0\|_{L^\infty},\|\nabla f_0\|_{L^\infty},M_{T},\varepsilon,\alpha)$ such that
\begin{align}\label{decay on phi}
\Phi(t)\le\frac{\Phi(0)}{[1+C_{\#}(\|f_0\|_{L^\infty},\|\nabla f_0\|_{L^\infty},M_{T},\varepsilon,\alpha)t]^\frac{4}{1+\alpha}}.
\end{align}
The result \eqref{decay on nabla f in L infty norm} then follows immediately from \eqref{decay on phi}.
\end{proof}
\begin{rem}
Under the assumptions \eqref{positivity assumption on f0}, \eqref{boundedness assumption on nabla f0} and \eqref{boundedness condition on nabla f in L1}, we can apply the results \eqref{decay on f in L infty norm} and \eqref{decay on nabla f in L infty norm} to further obtain
\begin{align}\label{decay on f in W 1,infty norm}
\|f\|_{W^{1,\infty}}(t)\le\frac{\|f_0\|_{W^{1,\infty}}}{(1+C_{*}t)^\frac{2}{1+\alpha}},
\end{align}
where $C_*=\min\{C(\|f_0\|_{L^\infty},\|f_0\|_{L^1},\alpha),C_{\#}(\|f_0\|_{L^\infty},\|\nabla f_0\|_{L^\infty},M_{T},\varepsilon,\alpha)\}>0$.
\end{rem}


\section{Convergence of solutions as $\alpha\to0^{+}$}\label{convergence for alpha section}

In this section, we address the convergence of solutions as $\alpha\to0^{+}$. To begin with, using Theorem~\ref{local existence thm} and Remark~\ref{uniform bound remark}, for $\alpha\in[0,\frac{1}{2})$, $k\ge4$ and $f_0\in H^k(\R^2)$, there exists $\bar{T}>0$ such that $f_\alpha$ and $f$ are functions in $C^1([0,\bar{T}];H^k(\R^2))$ satisfying
\begin{align*}
\frac{\partial f_{\alpha}(x,t)}{\partial t} &= \PV\int_{\mathbb{R}^2}  \frac{\left(\nabla f_{\alpha}(x,t) - \nabla f_{\alpha}(x-y,t)\right)y}{\left[|y|^{2} +\left( f_{\alpha}(x,t) - f_{\alpha}(x-y,t)\right)^2\right]^{\frac{3+\alpha}{2}}} \, dy \\
f_{\alpha}({x,0}) &= f_{0}(x)
\end{align*}
and
\begin{align*}
\frac{\partial f}{\partial t}(x,t) &= \PV\int_{\mathbb{R}^2}  \frac{\left(\nabla f(x,t) - \nabla f(x-y,t)\right)y}{\left[|y|^{2} +\left( f(x,t) - f(x-y,t)\right)^2\right]^{\frac{3}{2}}} \, dy \\
f(x,0) &= f_{0}(x)
\end{align*}
as well as the uniform bound \eqref{uniform bound on H4 norm} with $\bar{\alpha}=\frac{1}{2}$. We further define $A_\alpha$ and $A_{0}$ by
\begin{align*} 
A_{\alpha} &\triangleq A_{\alpha}(x,y,t) = |y|^{2}+(f_{\alpha}(x, t)-f_{\alpha}(x-y, t))^{2} \\
A_{0} &\triangleq A_{0}(x,y,t) = |y|^{2}+(f(x, t)-f(x-y, t))^{2}.
\end{align*}
The following theorem shows that $f_\alpha$ converges to $f$ in $L^2$ as $\alpha\to0^+$.
\begin{thm}\label{L2 convergence thm}
For $\alpha\in[0,\frac{1}{2})$, $k\ge4$ and $f_0\in H^k(\R^2)$, assume that $f_\alpha$ and $f$ are solutions in $C^1([0,\bar{T}];H^k(\R^2))$ as defined above. Then for all $t\in[0,\bar{T}]$, we have
\begin{align}\label{L2 convergence}
\lim_{\alpha\to0^+}\|f_\alpha - f\|_{L^2}(t)=0.
\end{align}
\end{thm}
\begin{proof}
Throughout this proof, $C_{\bar{M}}$ always denotes a positive constant that depends on $\bar{M}$ given in \eqref{uniform bound on H4 norm} but is independent of $\alpha$. 

Let $g_{\alpha}=f_{\alpha}(x, t)-f(x, t)$, then $g_{\alpha}$ satisfies
\begin{equation}\label{egn for g alpha}
\begin{split}
\frac{\partial g_{\alpha}}{\partial t}(x, t) & = \PV \int_{\mathbb{R}^{2}} \frac{\left(\nabla g_{\alpha}(x, t)-\nabla g_{\alpha}(x-y, t)\right) \cdot y}{A_{\alpha}^{\frac{3+\alpha}{2}}} d y \\
& \qquad+ \PV \int_{\mathbb{R}^{2}} \frac{(\nabla f(x, t)-\nabla f(x-y, t)) \cdot y}{A_{\alpha}^{\frac{3+\alpha}{2}}} d y \\
& \qquad- \PV \int_{\mathbb{R}^{2}} \frac{(\nabla f(x, t)-\nabla f(x-y, t)) \cdot y}{A_{0}^{\frac{3}{2}}} d y. 
\end{split}
\end{equation} 
Multiplying \eqref{egn for g alpha} by $g_{\alpha}$ and integrating with respect to $x$,
\begin{equation*} 
\begin{split}
 \frac{\partial}{\partial t} \int_{\mathbb{R}^{2}}\left|g_{\alpha}(x, t)\right|^{2} d x&=\PV \int_{\mathbb{R}^{2}} 2g_{\alpha}(x, t) \int_{\mathbb{R}^{2}} \frac{\left(\nabla g_{\alpha}(x, t)-\nabla g_{\alpha}(x-y, t)\right) \cdot y}{A_{\alpha}^{\frac{3+\alpha}{2}}} d y d x \\
&\,\,\,+\PV \int_{\mathbb{R}^{2}} 2g_{\alpha} \int_{\mathbb{R}^{2}}(\nabla f(x, t)-\nabla f(x-y, t)) \cdot y \cdot\left(\frac{1}{A_{\alpha}^{\frac{3+\alpha}{2}}}-\frac{1}{A_{\alpha}^\frac{3}{2}}\right) \\
&\,\,\,+\PV \int_{\mathbb{R}^{2}} 2g_{\alpha} \int_{\mathbb{R}^{2}}(\nabla f(x, t)-\nabla f(x-y, t)) \cdot y \cdot\left(\frac{1}{A_{\alpha}^\frac{3}{2}}-\frac{1}{A_{0}^\frac{3}{2}}\right) \\
&\triangleq \mathbf{E}_{1}+\mathbf{E}_{2}+\mathbf{E}_{3}. 
\end{split}
\end{equation*}
We estimate the terms $\mathbf{E}_{1}$, $\mathbf{E}_{2}$ and $\mathbf{E}_{3}$ as follows. First of all, we rewrite $\mathbf{E}_{1}$ as follows (we drop the variable $t$ for simplicity):
\begin{align*}
\mathbf{E}_{1}=\mathbf{E}_{1,1}+\mathbf{E}_{1,2}+\mathbf{E}_{1,3},
\end{align*}
where
\begin{align*}
\mathbf{E}_{1,1}&=\PV\int_{\mathbb{R}^{2}} 2g_{\alpha}(x) \int_{\mathbb{R}^{2}} \frac{\nabla g_{\alpha}(x) \cdot y}{A_{\alpha}^{\frac{3+\alpha}{ 2}} }\\
\mathbf{E}_{1,2}&=-\PV\int_{\mathbb{R}^{2}} 2g_{\alpha}(x) \int_{\mathbb{R}^{2}} \frac{g_{\alpha}(x)-g_{\alpha}(y)}{A_{\alpha}^{\frac{3+\alpha}{2}}} \\
\mathbf{E}_{1,3}&=\PV\int_{\mathbb{R}^{2}} 2(3+\alpha)g_{\alpha}(y) \int_{\mathbb{R}^{2}} (g_{\alpha}(x)-g_{\alpha}(y))\\
&\qquad\qquad\qquad\qquad\times\left[\frac{\left(f_{\alpha}(x)-f_{\alpha}(y)\right) \cdot\left(f_{\alpha}(x)-f_{\alpha}(y)-\nabla f_{\alpha}(y) \cdot(x-y)\right)}{A_{\alpha}^{\frac{5+\alpha}{2}}}\right]. 
\end{align*}
Upon integrating by parts, we have
\begin{align*}
\mathbf{E}_{1,1}&=\PV \int_{\mathbb{R}^{2}} (3+\alpha)|g_{\alpha}(x, t)|^{2}\int_{\mathbb{R}^{2}}\frac{\left(f_{\alpha}(x)-f_{\alpha}(x-y)\right)\left(\nabla f_{\alpha}(x)-\nabla f_{\alpha}(x-y)\right) \cdot y}{A_{\alpha}^{\frac{5+\alpha}{ 2}} }\\
&\le (3+\alpha)\left\|f_{\alpha}\right\|_{ C^{1}}\left\|g_{\alpha}\right\|_{L^{2}}^{2}+(3+\alpha) M_{\alpha}\left(f_{\alpha}\right)\left\|g_{\alpha}\right\|_{L^{2}}^{2}\\
&\le 4\left\|f_{\alpha}\right\|_{ C^{1}}\left\|g_{\alpha}\right\|_{L^{2}}^{2} + 4M_{\alpha}\left(f_{\alpha}\right)\left\|g_{\alpha}\right\|_{L^{2}}^{2}
\end{align*}
where $M_\alpha(\cdot)$ is given by \eqref{def of M alpha (f)}. Following the proof of Theorem~\ref{local existence thm}, Using the uniform bound \eqref{uniform bound on H4 norm}, we readily have
\begin{align*}
4\left\|f_{\alpha}\right\|_{ C^{1}}+4M_{\alpha}(f_{\alpha})\le C_{\bar{M}},
\end{align*}
and hence
\begin{align*}
\mathbf{E}_{1,1}\le C_{\bar{M}}\left\|g_{\alpha}\right\|_{L^{2}}^{2}.
\end{align*}
On the other hand, by making change of variables,
\begin{align*}
\mathbf{E}_{1,2}=-\frac{1}{2}\int_{\R^2}\int_{\R^2}\frac{(g_\alpha(x)-g_\alpha(y))^2}{A_{0}^\frac{3+\alpha}{2}}\le 0,
\end{align*}
while for $\mathbf{E}_{1,3}$, following the similar method for estimating $\mathbf{J}_3$ as in the proof of Theorem~\ref{local existence thm}, we have
\begin{align*}
\mathbf{E}_{1,3}\le 8\left[\left\|f_{\alpha}\right\|_{C^{2, \delta}}^{4}+\left\|f_{\alpha}\right\|_{C^{1}}\|f\|_{C^{2, \delta}}+1\right]\left\|g_{\alpha}\right\|_{L^{2}}^{2}\le C_{\bar{M}}\left\|g_{\alpha}\right\|_{L^{2}}^{2}.
\end{align*} 
Therefore we obtain
\begin{align}\label{bound on E1}
\mathbf{E}_{1}\le C_{\bar{M}}\left\|g_{\alpha}\right\|_{L^{2}}^{2}.
\end{align}
Next for the term $\mathbf{E}_2$, using mean value theorem and the uniform bound \eqref{uniform bound on H4 norm}, we have
\begin{align*}
\left|\frac{1}{A_{\alpha}^{\frac{3+\alpha}{2}}}-\frac{1}{A_{\alpha}^{3 / 2}}\right|&\le\left|-\frac{3+\alpha}{2}+\frac{3}{2}\right|\sup_{r\in(-\frac{3+\alpha}{2},-\frac{3}{2})}|A_{\alpha}^r||\log(A_\alpha)|\\
&\le\frac{\alpha}{2}|y|^{-3}\log(M^2)=\alpha|y|^{-3}\log(M).
\end{align*}
Therefore by applying the above estimate to $\mathbf{E}_2$, it implies
\begin{align}\label{bound on E2}
\mathbf{E}_2\le \alpha C_{\bar{M}}\left\|g_{\alpha}\right\|_{L^{2}}.
\end{align}
For the term $\mathbf{E}_3$, using mean value theorem and the uniform bound \eqref{uniform bound on H4 norm}, we also have
\begin{align*}
|A_{0}^\frac{3}{2}-A_{\alpha}^\frac{3}{2}|\le C_{\bar{M}}|g_{\alpha}(x)-g_{\alpha}(x-y)||y|^2,
\end{align*}
and hence
\begin{align}\label{bound on E3}
\mathbf{E}_3\le C_{\bar{M}}\left\|g_{\alpha}\right\|_{L^{2}}^{2}.
\end{align}
Combining \eqref{bound on E1}, \eqref{bound on E2} and \eqref{bound on E3}, we obtain
\begin{align}\label{full bound on g alpha}
\frac{\partial}{\partial t}\|g_{\alpha}\|^2_{L^2}\le 2C_{\bar{M}}\|g_{\alpha}\|^2_{L^2}+\alpha C_{\bar{M}}\|g_{\alpha}\|_{L^2}.
\end{align}
Applying Gr\"{o}nwall's inequality to \eqref{full bound on g alpha} and recalling that $g_{\alpha}(x,0)=0$, the results \eqref{L2 convergence} follows immediately by taking $\alpha\to0^{+}$.
\end{proof}
\begin{rem}\label{rem on Hk convergence}
For $1\le k\le 4$, by performing similar analysis on $\partial^k_{x_1} g_\alpha$ and $\partial^k_{x_2} g_\alpha$, for all $t\in[0,\bar{T}]$, we can also obtain
\begin{align*}
\lim_{\alpha\to0^+}\|\partial^k_{x_1}f_\alpha - \partial^k_{x_1}f\|_{L^2}(t)=\lim_{\alpha\to0^+}\|\partial^k_{x_2}f_\alpha - \partial^k_{x_2}f\|_{L^2}(t)=0.
\end{align*}
Hence we conclude that
\begin{align}\label{Hk convergence}
\lim_{\alpha\to0^+}\|f_\alpha - f\|_{H^k}(t)=0.
\end{align}
\end{rem}
Using Theorem~\ref{L2 convergence thm} and Remark~\ref{rem on Hk convergence}, we can further obtain $L^1$-convergence provided that the initial interface satisfies the condition \eqref{positivity assumption on f0}. The results are given in the following theorem.
\begin{thm}\label{L1 convergence thm}
For $\alpha\in[0,\frac{1}{2})$, $k\ge4$ and $f_0\in H^k(\R^2)$, assume that $f_\alpha$ and $f$ are solutions in $C^1([0,\bar{T}];H^k(\R^2))$ as defined above. If $f_0(x)\le0$ or $f_0(x)\ge0$ for all $x\in\R^2$, then for all $t\in[0,\bar{T}]$, we have
\begin{align}\label{L1 convergence}
\lim_{\alpha\to0^+}\|f_\alpha - f\|_{L^1}(t)=0.
\end{align}
\end{thm}
\begin{proof}
We give the proof for the case when $f_0\ge0$, and the case for $f_0\le0$ is just similar. By Theorem~\ref{max principle thm} and the assumption that $f_0\ge0$, for all $\alpha\in[0,\frac{1}{2})$, we have $f_{\alpha}\ge0$ and $f\ge0$. Also from the proof of Theorem~\ref{decay of f thm}, we know that
\begin{align*}
\int_{\R^2}\frac{\partial f_\alpha}{\partial t}(x,t)dx = \int_{\R^2}\frac{\partial f}{\partial t}(x,t)dx = 0,
\end{align*}
and hence $\|f_\alpha\|_{L^1}(t)=\|f\|_{L^1}(t)=\|f_0\|_{L^1}(t) $ for all $\alpha\in[0,\frac{1}{2})$ and $t\ge0$.

By applying the Sobolev embedding theorem, together with the uniform bound \eqref{uniform bound on H4 norm}, the convergence result \eqref{Hk convergence} implies
\begin{align*}
\lim_{\alpha\to0^+}\|f_\alpha - f\|_{L^\infty}(t)=0,
\end{align*}
which gives $f_\alpha(\cdot,t)\to f(\cdot,t)$ a.e. as $\alpha\to0^+$ for all $t\ge0$. 

Define $h_{\alpha}=|f_\alpha - f|$, then $h_{\alpha}\in L^1$ and we have $h_{\alpha}\to0$ a.e. as $\alpha\to0^+$. Observe that $|h_{\alpha}|\le|f_\alpha| + |f|$ and $|f_\alpha| + |f|\to 2|f|$ a.e. as $\alpha\to0^+$ with 
\begin{align*}
\lim_{\alpha\to0^+}\int_{\R^2}(|f_\alpha| + |f|)(x,t)dx=2\int_{\R^2}|f|(x,t)dx,\qquad\text{for all $t\ge0$.}
\end{align*}
Therefore using dominated convergence theorem (see for example \cite{Fo99}), it implies
\begin{align*}
\lim_{\alpha\to0^+}\int_{\R^2}h_{\alpha}(x,t)dx=0,\qquad\text{for all $t\ge0$,}
\end{align*}
and we conclude that \eqref{L1 convergence} holds for all $t\ge0$.
\end{proof}

\appendix

\section{Existence of $k_0(\alpha)$}\label{existence of k0 and estimates on principal values sec}

In this appendix, we show the existence of $k_0(\alpha)$ which appears in Theorem~\ref{global existence thm}. To begin with, we give an estimate for the principal value of integral involving $\frac{\exp(irS)}{r^{1+\alpha}}$ with $S\in\R$. More precisely, we prove the following lemma.

\begin{lemma}\label{estimates on PV lem}
Let $\mathbf{S}\in\R$ and $\alpha\in(0,1)$. There exists a positive constant $C(\alpha)>0$ which depends on $\alpha$ but is independent of $\mathbf{S}$ such that
\begin{align}\label{bound on PV} 
 \left|\PV\int_{\R}\frac{\exp(i r \mathbf{S})}{r|r|^{\alpha}} \right|\leq C(\alpha)|\mathbf{S}|^{\alpha}.
\end{align}
\end{lemma}
\begin{proof}
Notice that since $\cos(r\mathbf{S})/{r|r|^{\alpha}}$ is an odd function $r$ while $\sin(r\mathbf{S})/{r|r|^{\alpha}}$ is an even function in $r$, we readily have
\begin{align*}
 \PV\int_{\R}\frac{\exp(i r \mathbf{S})}{r|r|^{\alpha}}= 2\PV\int_{0}^{+\infty}\frac{\sin(r \mathbf{S})}{r|r|^{\alpha}}=2\PV\int_{0}^{+\infty}\frac{\sin(r \mathbf{S})}{r^{\alpha+1}}.
\end{align*}
Observe that the bound \eqref{bound on PV} holds trivially for $\mathbf{S}=0$, without loss of generality we assume that $\mathbf{S}\neq0$. Now for any $\varepsilon>0$, we split the integral as follows.
\begin{align}\label{approx integral PV}
\int_{0}^{+\infty}\frac{\sin(r \mathbf{S})}{r^{\alpha+1}}=\int_{0}^{\varepsilon}\frac{\sin(r \mathbf{S})}{r^{\alpha+1}}dr+\int_{\varepsilon}^{+\infty}\frac{\sin(r \mathbf{S})}{r^{\alpha+1}}dr.
\end{align}
For the first integral on the right side of \eqref{approx integral PV},
\begin{align*}
\left|\int_{0}^{\varepsilon}\frac{\sin(r \mathbf{S})}{r^{\alpha+1}}dr\right|&\le 2|S|\int_{0}^{\varepsilon}\frac{dr}{r^\alpha}\\
&\le 2|S|\frac{\varepsilon^{1-\alpha}}{1-\alpha},
\end{align*}
and for the second integral on the right side of \eqref{approx integral PV}, we perform a change of variable $u=r\mathbf{S}$ to obtain
\begin{align*}
\left|\int_{\varepsilon}^{+\infty}\frac{\sin(r \mathbf{S})}{r^{\alpha+1}}dr\right|=\left|\int_{\varepsilon S}^{+\infty}\frac{\sin(u)}{u^{\alpha+1}}\frac{1}{\mathbf{S}^{-(1+\alpha)}\mathbf{S}}du\right|\le |\mathbf{S}|^\alpha\frac{|\varepsilon\mathbf{S}|^\alpha}{1-\alpha}.
\end{align*}
We choose $\varepsilon=\mathbf{S}^{-1}$ and $C(\alpha)=\frac{6}{1-\alpha}$, then the bound \eqref{bound on PV} follows.
\end{proof}
\begin{rem}
As proved in \cite{constantin2016muskat}, for the case $\alpha=0$, we also have
\begin{align*}
\PV\int_{\R}\frac{\exp(i r \mathbf{S})}{r|r|^{\alpha}}=\pi i\sgn(\alpha). 
\end{align*}
Hence by choosing
\begin{align*}
C(\alpha)=\left \{
\begin{array}{lllll}
\frac{6}{1-\alpha},~~~~~&\alpha\in(0,1),\\
\pi,~~~~~&\alpha=0,\\
\end{array} \right.
\end{align*}
we can see that \eqref{bound on PV} is indeed valid for all $\alpha\in[0,1)$.
\end{rem}
Next, we show that the constant $k_0(\alpha)$ in Theorem~\ref{global existence thm} exists, which is stated in the following lemma.
\begin{lemma}\label{existence of k0 alpha lem}
For any $\alpha\in(0,1)$, there exists $k_0(\alpha)>0$ such that for all $0 \leq z < k_{0}(\alpha),$ we always have
\begin{equation}\label{condition for k(alpha) with x}
\left[ \frac{(10 + 4 \alpha) z^2 + (2 + \alpha) z^4}{(1 - z^2)^{\frac{7 + \alpha}{2}}} -1\right] < \frac{1}{2C(\alpha)},
\end{equation}
where $C(\alpha)>0$ is given in Lemma~\ref{estimates on PV lem}.
\end{lemma}
\begin{proof}
Notice that by the identity \eqref{power series for an}, we have
\begin{align*}
 \frac{(10 + 4 \alpha) z^2 + (2 + \alpha) z^4}{(1 - z^2)^{\frac{7 + \alpha}{2}}} -1 = \sum_{n=1}^{\infty}  \frac{\Gamma\left(\frac{3 + \alpha}{2} + n\right)}{\Gamma\left(\frac{3 + \alpha}{2}\right)} \cdot \frac{1}{n!} (2n + 1)^{2} z^{2n},
\end{align*}
and hence it suffices to show there exists $k_0(\alpha)>0$ such that 
\begin{align}\label{equiv form of k_0 alpha}
2C(\alpha)\sum_{n=1}^{\infty}  \frac{\Gamma\left(\frac{3 + \alpha}{2} + n\right)}{\Gamma\left(\frac{3 + \alpha}{2}\right)} \cdot \frac{1}{n!} (2n + 1)^{2} k_0(\alpha)^{2n}<1.
\end{align}
For each $N\in\mathbb{N}$, we first split the summation as follows.
\begin{align*}
2C(\alpha)\sum_{n=1}^{\infty}  \frac{\Gamma\left(\frac{3 + \alpha}{2} + n\right)}{\Gamma\left(\frac{3 + \alpha}{2}\right)} \cdot \frac{1}{n!} (2n + 1)^{2} k_0(\alpha)^{2n}=S_1+S_2,
\end{align*}
where $S_1$ and $S_2$ are given by
\begin{align*}
S_1&=2C(\alpha)\sum_{n=1}^{N-1}  \frac{\Gamma\left(\frac{3 + \alpha}{2} + n\right)}{\Gamma\left(\frac{3 + \alpha}{2}\right)} \cdot \frac{1}{n!} (2n + 1)^{2} k_0(\alpha)^{2n},\\
S_2&=2C(\alpha)\sum_{n=N}^{\infty}  \frac{\Gamma\left(\frac{3 + \alpha}{2} + n\right)}{\Gamma\left(\frac{3 + \alpha}{2}\right)} \cdot \frac{1}{n!} (2n + 1)^{2} k_0(\alpha)^{2n}.
\end{align*}
Since $\Gamma(x)\approx x^{x-\frac{1}{2}}e^{-x}$, there exists positive constants $c_1$ and $c_2$ such that for all $n\ge N$ with $N\in\mathbb{N}$,
\begin{align*}
\Gamma(\frac{3+\alpha}{2}+n)&\le c_1(\frac{3+\alpha}{2}+n)^{n+1+\frac{\alpha}{2}}e^{-(\frac{3+\alpha}{2}+n)},\\
c_2\sqrt{n}\cdot(\frac{n}{e})^n&\le n!.
\end{align*}
Therefore for each $N\in\mathbb{N}$, we have
\begin{align*}
S_2&\le \frac{2C(\alpha)}{\Gamma(\frac{3 + \alpha}{2})}\frac{c_1}{c_2}\sum_{n=N}^\infty(2n+1)^2(n+\frac{3+\alpha}{2})^{n+1+\frac{\alpha}{2}}e^{-(n+\frac{3+\alpha}{2})}\frac{1}{\sqrt{n}}(\frac{n}{e})^{-n}k_0(\alpha)^{2n}\\
&\le \frac{2C(\alpha)}{\Gamma(\frac{3 + \alpha}{2})}\frac{c_1}{c_2}\sum_{n=N}^\infty(2n+1)^2(n+\frac{3+\alpha}{2})^{1+\frac{\alpha}{2}}\frac{1}{\sqrt{n}}(1+\frac{3+\alpha}{2n})^n e^{-\frac{\alpha+3}{2}}k_0(\alpha)^{2n}.
\end{align*}
For sufficiently large $n$, we have $(1+\frac{3+\alpha}{2n})^n\approx e^{-\frac{\alpha+3}{2}}$, and $\sum_{n=1}^\infty n^p a^n<\infty$ for all $a\in(0,1)$ and $p\ge1$. Therefore, there exists $N\in\mathbb{N}$ such that for all $k_0(\alpha)\in(0,1)$, it implies that
\begin{align}\label{bound on S2}
S_2&\le \tilde C(\alpha)\sum_{n=N}^\infty n^{\frac{5}{2}+\frac{\alpha}{2}}k_0(\alpha)^{2n}<\frac{1}{2},
\end{align}
where $\tilde C(\alpha)$ is a positive constant which depends on $C(\alpha)$, $c_1$, $c_2$ and $\alpha$. We fix $N$ such that \eqref{bound on S2} holds, then there exists $k_0(\alpha)\in (0,1)$ such that $S_1$ can be bounded by
\begin{align}\label{bound on S1}
S_1\le 2k_0(\alpha)^2C(\alpha)\sum_{n=1}^{N-1}  \frac{\Gamma\left(\frac{3 + \alpha}{2} + n\right)}{\Gamma\left(\frac{3 + \alpha}{2}\right)} \cdot \frac{1}{n!} (2n + 1)^{2}<\frac{1}{2}.
\end{align}
Combining the bounds \eqref{bound on S2} and \eqref{bound on S2}, we can see that there exists $k_0(\alpha)\in (0,1)$ such that
\begin{align*}
S_1+S_2<1,
\end{align*}
hence the bound \eqref{equiv form of k_0 alpha} follows as well.
\end{proof}


{\section*{Acknowledgment}
We thank the anonymous referees for their helpful comments. Q. Khan and A. Suen are partially supported by Hong Kong General Research Fund (GRF) grant project number 18300622. B.Q. Tang is supported by NAWI Graz.}

\bibliographystyle{amsalpha}

\bibliography{References_for_active_scalar}

\end{document}